\documentclass[11pt]{amsart}
\usepackage{hyperref}

\usepackage{geometry} 
\usepackage{verbatim} 
\usepackage{color}
\usepackage{amsmath}
\usepackage{amsfonts, amssymb, amsthm, setspace, latexsym}
\usepackage{multicol}
\usepackage{graphicx}
\usepackage[at]{easylist}

\usepackage{calc}

\makeatletter
\def\textSq#1{%
\begingroup
\setlength{\fboxsep}{0.0ex}
\setbox1=\hbox{#1}
\setlength{\@tempdima}{3.3ex}
\setlength{\@tempdimb}{(\@tempdima-1ex)/2}
\raise-\@tempdimb\hbox{\fbox{\vbox to \@tempdima{%
  \vfil\hbox to \@tempdima{\hfil\copy1\hfil}\vfil}}}%
\endgroup%
}
\def\squarebox#1{\textSq{\ensuremath{#1}}}%
\makeatother

\usepackage{soul}  
\numberwithin{equation}{section}
\newtheorem*{theorem*}{Theorem}
\newtheorem{theorem}{Theorem}[section]

\newtheorem{proposition}[theorem]{Proposition}

\newtheorem{corollary}[theorem]{Corollary}
\theoremstyle{definition}
\newtheorem{definition}[theorem]{Definition}
\newtheorem{remark}[theorem]{Remark}
\newtheorem{example}[theorem]{Example}
\newtheorem{def/prop}[theorem]{Definition/Proposition}

\usepackage{todonotes}
\newcommand{\NW}{NW }
\newcommand{\va}[1]{v_{#1}}

\newcommand{\mf}{\mathfrak}
\newcommand{\Y}{\mathcal{Y}}
\newcommand{\YDs}{\operatorname{YD}}
\newcommand{\YDg}{\operatorname{YD}}
\newcommand{\YDgs}{\operatorname{YD}}
\newcommand{\Yt}{\widetilde{Y}} 
\newcommand{\cYt}{\widetilde{\Y}} 
\newcommand{\X}{\mathcal{X}}
\newcommand{\cZ}{\mathcal{Z}}
\newcommand{\Z}{\mathcal{Z}}
\newcommand{\YZ}{\Y\text{-}\Z}
\newcommand{\M}{M}
\newcommand{\N}{N}
\newcommand{\fun}{\ell}
\newcommand{\q}{\mathtt{q}}
\newcommand{\Sq}{\boldsymbol{q}} 
\newcommand{\brT}{\mathsf{T}} 
\newcommand{\brY}{\mathsf{Y}} 
\newcommand{\brX}{\mathsf{X}} 
\newcommand{\Q}{Q} 
\newcommand{\K}{\mathcal{K}}
\newcommand{\Kt}{{\mathcal{K}}^{\times}}
\newcommand{\supp}{\operatorname{supp}}
\newcommand{\AffSym}{\widehat{\mathfrak{S}}_n}
\newcommand{\AffSymSL}{\overline{\mathfrak{S}}_n}
\newcommand{\Sym}{{\mathfrak{S}}}
\newcommand{\fsum}{{\mathfrak{s}}} 
\newcommand{\Yprod}{\boldsymbol{r}} 
\newcommand{\Sn}{{\Sym}_n}
\newcommand{\SN}{{\Sym}_N}
\newcommand{\id}{\operatorname{Id}}
\newcommand{\act}{\operatorname{act}}
\newcommand{\e}[1]{\varepsilon_{#1}}
\newcommand{\ep}[1]{\epsilon_{#1}}

\newcommand{\epi}{\ep{i}}
\newcommand{\uu}{\underline{u}}
\newcommand{\uz}{\underline{z}}
\newcommand{\diag}{\mathrm{diag}}
\newcommand{\diagRi}[2]{\diag_{{#1}}({#2})}

\newcommand{\cH}{\mathcal{H}}

\newcommand{\cO}{\mathcal{O}}

\newcommand{\Per}{\mathcal{P}\!er}
\newcommand{\Tab}{\mathcal{T}\!ab}
\newcommand{\STab}{\overline{\mathcal{T}\!ab}}
\newcommand{\SPer}{\overline{\Per}} 

\newcommand{\T}{\mathcal{T}} 
\newcommand{\cT}{\mathcal{T}}
\newcommand{\Sk}[3]{\mathcal{SK}^{#1,#2}_{#3}} 
\newcommand{\Skk}{\Sk{N}{k}{\lambda}}
\newcommand{\LWalk}[3]{\mathcal{W}^{#1,#2}_{#3}}
\newcommand{\DNk}[3]{\mathrm{D}^{#1,#2}_{#3}} 
\newcommand{\Dlam}{\DNk Nk{\lambda}}
\newcommand{\sLatp}{\Lambda^+_{\slN}}
\newcommand{\gLatp}{\Lambda^+_{\glN}} 
\newcommand{\sLat}{\Lambda_{\slN}}  
\newcommand{\gLat}{\Lambda_{\glN}} 
\newcommand{\gsLatp}{\Lambda^+} 

\newcommand{\wdet}{\mathbf{d}} 

\newcommand{\MkN}{M(k^N)} 
\newcommand{\sMkN}{\overline{M}(k^N)} 
\newcommand{\LkN}{L(k^N)} 
\newcommand{\sLkN}{\overline{L}(k^N)} 
\newcommand{\sLkNa}[1]{\left(\overline{L}(k^N)\right)^{#1}} 
\newcommand{\wtR}{\mathrm{wt}}
\newcommand{\syt}{\mathrm{SYT}}
\newcommand{\sytk}{\syt(k^N)}
\newcommand{\psyt}{\mathrm{P}_n\mathrm{SYT}}
\newcommand{\psytk}{\psyt(k^N)}
\newcommand{\SLpsyt}{\psyt(k^N)\big/\pi^n}
\newcommand{\psytn}[1]{\mathrm{P}_{#1}\mathrm{SYT}}
\DeclareMathOperator{\Rect}{Rect}
\newcommand{\rect}{\Rect(N,k)} 
\newcommand{\prel}{\sim} 
\newcommand{\bR}{{\overline{R}}} 

\newcommand{\ZZ}{\mathbb{Z}}
\newcommand{\R}{\mathbb{R}}
\newcommand{\cD}{\mathcal{D}}
\newcommand{\Dq}{\cD_\q(G)}
\newcommand{\Oq}{\cO_\q(G)}
\newcommand{\Uq}{U_\q(\g)}
\newcommand{\detq}{\operatorname{det}_\q}
\newcommand{\cE}{\mathbb{E}} 
\newcommand{\calE}{\mathcal{E}} 
\newcommand{\cR}{\mathcal{R}}
\newcommand{\cU}{U}
\newcommand{\g}{\mathfrak{g}}
\newcommand{\HH}{\mathbb{H}}
\newcommand{\HG}{\HH_{q,t}(\GL_n)}
\newcommand{\HS}{\HH_{\Sq,t}(\SL_n)}
 
\newcommand{\CC}{\mathbb{C}}

\newcommand{\Rep}{\operatorname{Rep}}

\newcommand{\ot}{\otimes}
\newcommand{\modu}{\textrm{-mod}}
\newcommand{\Hom}{\operatorname{Hom}}
\newcommand{\SL}{SL}
\newcommand{\SLN}{\SL_N}
\newcommand{\GL}{GL}
\newcommand{\GLN}{\GL_N}
\newcommand{\slN}{\mathfrak{sl}_N}
\newcommand{\glN}{\mathfrak{gl}_N}

\DeclareMathOperator{\Res}{Res}
\DeclareMathOperator{\Ind}{Ind}

\newcounter{r}
\newcounter{c}

\title[Rectangular representations and Schur-Weyl duality]{
The rectangular representation of the double affine Hecke algebra via elliptic Schur-Weyl duality}
\author{David Jordan and Monica Vazirani}
\date{\today}

\begin{document}
\begin{abstract}
Given a module $M$ for the algebra $\cD_\q(G)$ of quantum differential operators on $G$, and a positive integer $n$, we may equip the space $F_n^G(M)$ of invariant tensors in $V^{\ot n}\ot M$, with an action of the double affine Hecke algebra  of type $A_{n-1}$. Here $G= \SLN$ or $\GLN$, and $V$ is the $N$-dimensional defining representation of $G$.

In this paper we take $M$ to be the basic $\cD_\q(G)$-module, i.e. the quantized coordinate algebra $M= \cO_\q(G)$.  We describe a weight basis for $F_n^G(\cO_\q(G))$ combinatorially in terms of walks in the type $A$ weight lattice, and standard periodic tableaux, and subsequently identify $F_n^G(\cO_\q(G))$ with the irreducible ``rectangular representation" of height $N$ of the double affine Hecke algebra.
\end{abstract}

\maketitle
\tableofcontents

\section{Introduction}
Associated to the quantum group $U_\q(\g)$ is the quantized coordinate algebra $\cO_\q(G)$ and the algebra of quantum differential operators $\cD_\q(G)$.  The algebra $\cO_q(G)$ is deformation of $\cO(G)$ along the so-called Semenov-Tian-Shansky Poisson bracket, while $\Dq$ is a simultaneous $\q$-deformation of the algebra $D(G)$ of differential operators on $G$, and of functions on $G\times G$, with respect to the Heisenberg double Poisson bracket.  The quantized coordinate algebra $\cO_\q(G)$ is naturally a module for $\Dq$, which we call the basic $\cD_\q(G)$-module.  In this paper, we will be concerned with the cases $G=GL_N$ and $G=SL_N$.

Associated to the type $A_{n-1}$ root system are Cherednik's double affine Hecke algebras\footnote{Because the parameters $\q$ in $U_\q(\g)$, $q$ in $\HH_{q,t}(GL_N)$ and $\Sq$ in $\HH_{\Sq,t}(SL_n)$ do not coincide, yet each notation is well-established, we distinguish them by using different typeface.} $\HH_{q,t}$, which we will also refer to as the DAHA.
  There is a variant for $G=GL_n$, and for $G=SL_n$; we will consider both.  Each is a universal deformation of the semi-direct product, $\CC[\Sn]\rtimes \cD_\q(H)$, of the group algebra of the symmetric group $\Sn$ and the algebra of \emph{difference operators} on the Cartan subgroup $H \subset G$.  The $t$ parameter deforms the reflection action of $\CC[\Sn]$ to an action of the finite Hecke algebra, while the parameter $q$ corresponds to the step for the difference action on $H$.

This paper concerns a sort of ``elliptic" or ``genus one" Schur-Weyl duality functor, which relates the representations of $\cD_\q(GL_N)$ and $\cD_\q(SL_N)$ to the representations of the double affine Hecke algebras of type $A_{n-1}$, their $GL_n$ and $SL_n$ variants, respectively.  This kind of duality is very useful, as the representation theory of double affine Hecke algebras is well-understood in terms of type $A$ algebraic combinatorics, while the representation theory of $\cD_\q(G)$ is much less well-understood.  Here $n$ and $N$ are independent integer parameters, though in this paper we restrict to the case where $n=kN$, for some positive integer $k$.  In \cite{Jordan2008}, a construction was given of functors,
\begin{equation} F_n^{\SL}: \cD_\q(SL_N)\modu \to \HH_{\Sq,t}(SL_n)\modu,\label{Fn-intro}\end{equation}
from the category of $\cD_\q(SL_N)$-modules to the category of modules for $\HH_{\Sq,t}(SL_{n})$,
 specialized at $t=\q$, $\Sq=\q^{1/N}$.  The underlying vector space on which $\HH_{\Sq,t}(SL_n)$ acts is:
\begin{equation}F^{SL}_n(M) = (V^{\ot n} \ot M)^{inv},\label{Fn-formula}\end{equation}
of $U_\q(\slN)$-invariants in the tensor product, where we make $M$ a module over $U_\q(\mathfrak{sl}_N)$ via the quantum moment map, and we take invariants with respect to that action.  The action of $\HS$ is given by certain explicit operators, which we recall in Section \ref{sec-functor}.

The functor $F_n^{\SL}$ has a natural modification to a functor
$$F^{\GLN}_n: \cD_\q(GL_N)\modu \to \HH_{q,t}(GL_n)\modu,$$
where the parameters are specialized at
$t=\q$, $q=\q^{-2k}$. 

While the $SL$ version of the functor is simpler to define, the output is simpler to analyze in the $GL$ case.  The main results are essentially the same in the two cases, but with subtle minor modifications.  For this reason, in this paper we construct and analyze the functor in both settings in parallel, paying special care to what happens when we pass between the two settings.

Our main results all concern the case when $M$ is the basic $\cD_\q(G)$-module $M=\cO_\q(G)$; we compute the outputs $F^{\GLN}_n(\cO_\q(GL_N))$ and $F^{\SLN}_n(\cO_\q(SL_N))$ of the functor in this case.  This is easily seen to lie in Category $\cO$ for $\HH_{q,t}$, a particularly nice subcategory of $\HH_{q,t}$-modules, upon which a commutative subalgebra of $\HH_{q,t}$ acts locally finitely (see Section \ref{sec-Yssl}).  Our main result (given in more detail
in Theorem \ref{thm-Fn-isom}) is:

\begin{theorem*}
We have isomorphisms:
$$F_n^{GL}(\cO_\q(GL_N))\cong \LkN,\qquad F_n^{SL}(\cO_\q(SL_N))\cong \sLkN,$$
where $\LkN$ and $\sLkN$ denote the so-called rectangular representations of $\HG$ and $\HS$.
\end{theorem*}
\noindent We note that if $k= \frac nN$ is not a positive integer, then $F^G_n(\Oq) = 0$, see Remark \ref{rmk-kint}.

First, let us discuss the $G=GL$ case.  The rectangular representations are irreducible modules in Category $\cO$
which are moreover $\Y$-semisimple, where $\Y$ is the 
commutative subalgebra appearing in the definition of Category $\cO$,
see Section \ref{sec-rect}.
 The proof of the main theorem relies on identifying the weights of each module and then appealing to the structure of $\Y$-semisimple modules, to yield the isomorphism. In particular, these modules have simple spectrum
and so come with distinguished basis up to scaling. This lets us focus on a combinatorial analysis of the $\Y$-weights to
fully understand the module.

The computation of the $\Y$-weights involves three well-known elements of type $A$ combinatorics:  walks on the weight lattice, tableaux on skew diagrams, and periodic tableaux on an $N\times \infty$ strip.  Schur-Weyl duality considerations, together with the Peter-Weyl decomposition, lead us naturally to a $\Y$-weight basis of $F^{GL}_n(\cO_\q(GL_N))$ indexed by so-called \emph{looped walks} in the dominant chamber of the weight lattice.  In fact, each Peter-Weyl component is naturally a simple submodule upon restriction to the affine Hecke algebra.

We compute the $\Y$-weights, following a technique of Orellana and Ram \cite{OR}, involving well-known formulas for the action of the ribbon element of the quantum group.  The most natural combinatorial expression for the $\Y$-weights comes after we identify the set of looped walks with the set of standard tableaux on size $n$ skew diagrams with $k$ boxes in each of its $N$ rows. 
 Via this bijection, the $\Y$-weights can be read off as the values of certain ``diagonal labels" on the tableaux (a variant of what is sometimes called ``content" in the literature).

Finally, we introduce a procedure called ``periodization", which turns a tableau on such a skew diagram $D$ into a periodic tableau on the $N\times \infty$ strip.  Roughly, $D$ gets identified with a fundamental domain in the $N\times \infty$ strip, under horizontal shifting, and then the periodization map sends a standard tableau on $D$ to a standard periodic tableau on the $N\times \infty$ strip, by extending the entries according to a periodicity rule with respect to the shifting action. 
 We appeal to Cherednik's classification \cite{Cherednik03} of
irreducible $\Y$-semisimple $\HG$-modules via periodic Young diagrams, and
in particular to the enumeration of the $\Y$-weights via periodic tableaux
on such diagrams, described combinatorially in the follow-up paper,
\cite{SV}.  In this framework, the set of standard periodic tableaux on the
$N\times \infty$ strip index a weight basis for a unique irreducible
representation $\LkN$, hence allowing us to prove the theorem.

For $G=\SL$, the above story holds with minor modifications.  The role of the commutative algebra $\Y$ is played instead by an algebra $\Z$, which we may regard as a quotient of $\Y$ by a determinant-equals-one relation.
We consider periodic tableaux on the $N\times \infty$ strip modulo a natural equivalence relation of horizontal shifting by $k$ steps.  The diagonal labelling of these must be modified so as to be compatible with this periodicity.  Once this is done, analogous bijections hold as in the $GL$ case, and the $\Z$-weights can once again be read off from the (modified) diagonal labelling.
Cherednik's classification no longer holds on the nose, as the construction does not distinguish irreducible $\Z$-semisimple modules from those obtained by scaling the action of the shifting generator $\pi$ by a root of unity. Once we account for this, the $SL$ statement in the theorem follows.

\subsection{Rational degeneration of $F^{SL}_n$}
The functor $F^{SL}_n$ is a $\q$-deformation of a similar functor,
$$\mathrm{F}_n: D(\mathfrak{sl}_N)\modu \to \mathrm{RCA}_n(c)\modu,$$
introduced by Calaque, Enriquez and Etingof in \cite{CEE}.  Here $\mathrm{RCA}_n(c)$ denotes the rational Cherednik algebra of type $A_{n-1}$  with parameter $c=\frac{N}{n}$.  Their functor in turn builds on a similar construction
of Arakawa and Suzuki \cite{AS}
for the degenerate affine Hecke algebra. 
In \cite[Theorem 8.8]{CEE}, 
the authors compute the image of their functor on the basic $\cD$-module $M=\cO(\slN)$, and identify it as a unique simple quotient of an induced rectangular module.  The techniques we use here in the non-degenerate case are, however, completely different than in the degenerate setting.

One can connect these theorems more directly,
following \cite[Section 6]{Jordan2008} and \cite[Section 6] {BrochierJordan2014}.
  One can define a suitable degeneration $\cD_\q(SL_N)\leadsto D(\slN)$,
 and a degeneration $\HH_{q,t}\leadsto \mathrm{RCA}_n(c)$, compatible in such a way that we obtain a degeneration of the functors $F^{SL}_n\leadsto \mathrm{F}_n$.  Tracing through these degenerations, one may obtain from Theorem \ref{thm-Fn-isom} an independent proof Calaque, Enriquez, and Etingof's description of $\mathrm{F}_n(\cO(\slN))$.  

Category $\cO$ for $\mathrm{RCA}_n(c)$ is a highest weight category and its
simples can be obtained as unique simple quotients of induced modules.
While the construction of the DAHA-module $\LkN$
(resp. $\sLkN$) combinatorially via its weight basis was a key ingredient in
the proof of the main theorems, we also give an alternate construction as a
quotient of an induced module, motivated by the parallel construction
for the rational Cherednik algebra.  Another motivation for this alternate construction is that $\cO_\q(G)$ is most naturally constructed as an induced $\cD_\q(G)$-module.

More precisely,
we construct $\LkN$ as the unique simple
quotient of the module $M(k^N)$, induced from the $N\times k$ rectangular
representation for the affine Hecke algebra
in Corollary \ref{cor-LkN}.
In Theorem \ref{thm-sMkN}, we modify the construction of $\sLkN$ as a quotient
of the module $\sMkN$, induced from the $N\times k$ rectangular representation
for the appropriate quotient $H(\Z)$ of the affine Hecke algebra.
It is no longer the unique simple quotient, as $\sMkN$ also
has quotients that are obtained from $\sLkN$ by twisting
by an automorphism of $\HS$, see also Remark \ref{rem-SLsub}.

\subsection{Relation to factorization homology}

The action of $\HS$ on $F^{\SL}_n(M)$
 was initially defined via generators and relations. Following \cite{BBJ1,BBJ2}, it may be re-cast as a consequence of a much more general construction of the algebra $\cD_\q(G)$,
 via the factorization homology of surfaces valued in braided tensor categories.  Since such a description does not yet appear explicitly in the literature, we outline it here.  This context is not technically necessary for any of the proofs in this paper, so we will keep the discussion informal.

Recall from \cite{BBJ1}, we have the following equivalence of categories:
$$\int_{T^2\backslash D^2} \Rep_\q(SL_N) \simeq \cD_\q(SL_N)\modu_{\Rep_\q SL_N},$$
where the left hand side denotes the factorization homology of a once-punctured torus with coefficients in the braided tensor category $\Rep_\q(SL_N)$ of integrable $U_\q(\mathfrak{sl}_N)$-modules, and the right hand side denotes the category of $\cD_\q(SL_N)$-modules equipped with an equivariance structure.  In analogy with the theory of $D$-modules, these are called ``weakly equivariant" $\cD_\q(SL_N)$-modules.

Recall from \cite{BBJ2}, we have a further equivalence of categories:
$$\int_{T^2} \Rep_\q(SL_N) \simeq \cD_\q(\frac{SL_N}{SL_N})\modu,$$
where now the righthand side denotes the category of $\cD_\q(SL_N)$-modules such that the quantum adjoint action, obtained via the ``quantum moment map,"
$$\mu_\q:U_\q(\mathfrak{sl}_N)^{lf}\to\cD_\q(SL_N),$$ makes $\cD_\q(SL_N)$ into an integrable $U_\q(\mathfrak{sl}_N)$-module.  Here $U_\q(\slN)^{lf}$ denotes the sub-algebra consisting of elements which are locally finite for the quantum adjoint representation,
 see Section \ref{sec-quantum} for more details.  In analogy with the theory of $D$-modules, these are called ``strongly equivariant" $\cD_\q(SL_N)$-modules: a weakly equivariant $\cD_\q(G)$-module is strong if its weak equivariance structure coincides with that coming from the quantum moment map.  Alternatively, a $\cD_\q(G)$-module without a specified equivariance structure is a strong $\cD_\q(G)$-module if the $U_\q(\g)$-action given by the quantum moment map is integrable.

It is an immediate consequence of formula \eqref{Fn-formula} that $F_n^{\SL}$ factors through the functor $M\to M^{lf}$, of taking locally finite vectors for the quantum adjoint action, hence we may just as well regard it as a functor from $\cD_\q(\frac{SL_N}{SL_N})\modu$.

The $\HS$ action in the definition of $F^{SL}_n$ most naturally arises as an action of the elliptic braid group, which descends to its quotient $\HS$ in important special cases.  Let us now explain the origins of the elliptic braid group action in factorization homology terms.  Given an embedding $\iota^n: (D^2)^{\sqcup n}\to T^2$, i.e. a collection of $n$ disjoint discs in the torus, factorization homology produces a functor,
$$\iota^n_*:\Rep_\q(SL_N)^{\boxtimes n}\to \int_{T^2}\Rep_\q(SL_N)\simeq \cD_\q(\frac{SL_N}{SL_N})\modu.$$
This functor moreover carries an action, by natural isomorphisms, of the elliptic
 braid group $B_n^{Ell}$.
  The functors $\iota^n_*$ have been identified in \cite{BBJ1} and \cite{BBJ2} with free module functors to their respective categories of $\cD_\q(SL_N)$-modules.  It follows from these descriptions that, for any strongly equivariant $\cD_\q(SL_N)$-module $M$, we may identify
$$F^{SL}_n(M) \cong \Hom\!\left(\iota^n_*\left((^*V)^{\boxtimes n}\right),M\right),$$
where $V\in\Rep_\q(SL_N)$ denotes the defining representation, and $(^*V)$ denotes its left dual.  The action of the torus braid group $B_n^{Ell}$ -- and hence of $\HS$ -- on $F^{SL}_n(M)$ is that induced by the action of $B_n^{Ell}$ on the functor $\iota^n_*$ itself.

From this perspective the $\cD_\q(G)$-module $\cO_\q(G)$ can be understood to arise from the solid torus 3-manifold $D^2\times S^1$, with boundary the two-dimensional torus.  For this reason $\cO_\q(G)$ can be expected to play an important role in the associated four-dimensional topological field theory, more precisely in computing the invariants associated to 3-manifolds obtained by surgery.

In the $GL$ case, we modify the construction above by instead fixing an inclusion $\iota^{n+1}$, of $n+1$ discs into the torus.  Instead of applying $\iota^n_*$ to $V\boxtimes\cdots\boxtimes V$, we apply $\iota^{n+1}_*$ to $(\detq(V)^*)^{\ot k} \boxtimes V\boxtimes\cdots\boxtimes V$, where $\detq(V) := \bigwedge\nolimits_\q^N(V)$.  Instead of $B^{Ell}_{n}$, the relevant braid group becomes $B^{Ell}_{n,1}$.  See Section \ref{sec-elliptic-braid-group}.

\smallskip

In future work, we plan to extend the results of the present paper to
understand the structure of $\cD_\q(\frac {G}{G})$-mod more generally,
through the functor $F_n^G$.  As a starting point, we will introduce
the so-called quantum Springer sheaf, of which $\Oq$ is a distinguished
simple quotient, and in fact is a simple summand occuring with
multiplicity one.

\subsection{Outline}
In Section \ref{sec-typeA}, we collect some preliminaries about the weight lattices for type $A$, their $GL$ and $SL$ variants: in particular, the set of walks in the dominant chamber, and a bijection  to an appropriate set of standard skew tableaux.
 In Section \ref{sec-quantum}, we recall the quantum algebras we will study: $U_\q(\g)$, $\cO_\q(G)$, and $\cD_\q(G)$, and we recall some facts from the representation theory of $\Uq$.

In Section \ref{sec-DAHA}, we recall the elliptic braid group, the double affine Hecke algebras of $GL$ and $SL$ type, and how these all relate to one another.  In Section \ref{sec-Yssl}, we recall some structural results about Category $\cO$ for the double affine Hecke algebra, and in particular the $\Y$-semisimple modules.  In Section \ref{sec-rect}, we recall the classification of irreducible $\Y$-semisimple modules in Category $\cO$ in terms of periodic skew diagrams and their bases, indexed by standard tableaux; for expedience of exposition, we only focus on the case we will need, of the so-called ``rectangular representations."  We also discuss the construction of the rectangular representation as a quotient of an induced module.  In Section \ref{sec-main-results}, we state and prove our main results:  we give the GL modification of the functor $F^G_n$, we identify a weight basis for $F^G_n(\cO_\q(G))$, for $G=GL_N$ and $G=SL_N$, with the set of periodic standard tableaux on an infinite strip, and we subsequently identify $F^G_n(\cO_\q(G))$ as an $\HH_{q,t}$-module with the rectangular representation.

\subsection{Acknowledgements}
The authors wish to thank Pavel Etingof and Peter Samuelson for helpful discussions about the $GL$-modification of the functor $F_n^{\SL}$, Siddhartha Sahi for suggesting we extend results from $\SL$ to $\GL$ in the first place, and Stephen Griffeth for pointing out that Theorem \ref{thm-usqGL} might not hold for SL.
This project has received funding from the European Research Council (ERC) under the European Union's Horizon 2020 research and innovation programme (grant agreement no. 637618),
and from the Simons Foundation.  This collaboration began at the MSRI program, ``Geometric Representation Theory," in 2014, funded by NSF Grant No. DMS-1440140; we are grateful to MSRI for their hospitality.

\section{Combinatorics in type $A$: lattice walks and skew tableaux}
\label{sec-typeA}
In this section we review a number of well-known combinatorial constructions related to the weight lattice in type $A$.
\subsection{The $\GLN$ root system}
Let $\mathfrak{g}= \glN =\mathfrak{gl}(N,\mathbb{C})$
be the Lie algebra of the
general linear algebraic group $G=\GLN =\GL(N,\mathbb{C})$. 
Let
$\cE_{\glN} =\mathbb{R}^N$, with standard basis,
$$\mathcal{E} = \mathcal{E}_{\glN} =
\{\epi \quad | \quad i=1,\ldots, N\}.$$
and  symmetric form $\langle \, , \, \rangle$ with respect to
which $\mathcal{E}$ is an orthonormal basis\footnote{ 
Technically we should be using the pairing $\cE^* \times \cE
\to \R$, particularly if we were to consider non-simply laced type.
But for convenience we identify $\cE^*$ with $\cE$ and express
our formulas with respect to the symmetric form  we thus
chose to denote
$\langle \; , \; \rangle$ over the more standard 
 $( \; \mid \; )$.
}.
The weight lattice of $\glN$ is
$$\gLat = \bigoplus_{i=1}^N \ZZ \epi = \ZZ^N.$$
Elements of $\gLat$ are called integral weights. The dominant integral weights are 
$$\gLatp =
\{ m_1\ep{1} + \cdots + m_N\ep{N}\,\, |\,\,
m_i\in\mathbb{Z}, m_1\geq \cdots \geq m_N\}.$$

Let us denote $\wdet := \ep{1} + \ep{2} + \cdots + \ep{N}$,
which is the weight of trace representation in $\glN$,
corresponding to the
the character of the determinant representation of $\GLN$.

\begin{definition} Given a dominant integral weight $\lambda = \sum_i m_i \ep{i} \in \gLatp$
we denote by $\YDg(\lambda)$ the diagram (or  integer partition) with fewer than $N$ parts,
$$\YDg(\lambda) = (m_1-m_N, m_2-m_N, \ldots,  m_{N-1}-m_N, 0).$$
\end{definition}

The construction of $\YDg(\lambda)$ highlights a representative of $\lambda
\bmod \wdet$ with $m_N=0$.  We can recover $\lambda$ by remembering an
additional datum, that of a \emph{diagonal labelling} on $\YDg(\lambda)$.  The
diagonals of a Young diagram run down and to the right through
the NW and SE corners of its boxes; and a diagonal
labelling is an assignment of an integer to each diagonal.  The diagonal
through the upper left box of $\YDg(\lambda)$ is called the \emph{principal
diagonal}, and we decree that this diagonal is labelled with $m_N$.  The other
diagonals are labelled consecutively, so that the next diagonal to the right
is labelled $m_N+1$, etc.  Equivalently, we can say that the upper left box is
in row $1$ and column $m_N+1$, and then the diagonal is the column number
minus the row number.  These labels will be discussed further in Section
\ref{sec-walks-to-skew}.

Note the diagram $\YDg(\lambda+r \wdet)$ is
that of $\YDg(\lambda)$  shifted $r$ units right, and so
its diagonal labels are incremented $+r$.
Hence, although we draw the same diagram for $\lambda$
as well as $\lambda + r\wdet$, they are distinguished by their diagonal labels.

Note that if we think of the ``main" diagonal
of an integer partition $\gamma = (m_1, \ldots, m_N)$ with $m_N \ge 0$
as the diagonal through the box in its upper leftmost corner,
that is the one in its first row and first column, then
we traditionally assign its label or ``content"  to be $0$.
This is consistent with our labeling as
the  $(m_N+1)$st column of $\gamma$ agrees with the leftmost,
i.e. $(m_N+1)$st, column of $\YDg(\sum_i m_i \ep{i})$.
Or in other words the first column of  $\YDg(\sum_i m_i \ep{i})$
is also the first column of $\gamma$.
Further, our conventions let us make sense of partitions with negative parts.
See Figure \ref{fig-psiexample}.

Given $\lambda = \sum_i m_i \ep{i}$, its dual weight is $\lambda^* := \sum_i -m_i \ep{N+1-i}$, in terms of coordinate vectors $(m_1,\ldots,m_N)^* =(-m_N, \ldots, -m_1)$. Observe therefore that if one takes $\YDg(\lambda^*)$ and rotates it 180 degrees, then it is the complement to $\YDg(\lambda)$ in a $N\times (m_1 - m_N)$ rectangle. See Figure \ref{fig-dual-lambda}.

Let us describe the diagonal labels in terms of the inner product on $\gLatp$.
Consider $\lambda$ as compared to $\lambda + \ep{i}$. The diagram
has one extra box and we claim the diagonal of that box is labeled
\begin{equation}
\label{eq-diagonal-gl}
\langle  \lambda, \ep{i} \rangle +1-i
= \langle  \lambda+\ep{i}, \ep{i} \rangle -i.
\end{equation}
The new box in the $i$th row. Note that $m_i = \langle  \lambda, \ep{i}
\rangle $. 
The $i$th row of $\lambda$ has ``length" $m_i -m_N$, 
which is to say it ends $m_i-m_N$ units to the right of the leftmost
column,
so the new box is in column $m_i+1 = (1+m_N) + (m_i-m_N)$, yielding it is
on the $m_i+1-i $ diagonal.

We introduce a special weight $\rho$ given by
$$ \rho = \frac 12
((N-1)\ep{1} + (N-3)\ep{2} + (N-5)\ep{3} + \cdots + (1-N) \ep{N}).$$
Observe $2\rho  \in \gLatp$, although $\rho$ might not be
depending on the parity of $N$.
We also note
\begin{equation}
\label{eq-form-on-rho-gl}
\langle 2\rho,\ep{i}  \rangle = N +1-2i.
\end{equation}
Hence another way to describe the diagonal label of the box above
is by 
\begin{equation}
\label{eq-form-diag-gl}
m_i+1-i= \langle  \lambda, \ep{i} \rangle +
\langle \rho,\ep{i}  \rangle -
\langle \rho,\ep{1}  \rangle.
\end{equation}

\begin{remark}
We remark that $\mathcal{E}$ are the weights of the $N$-dimensional
defining representation $V = V_{\ep{1}}$ of $G$.
\end{remark}

\subsection{The $\SL_N$ root system}
Let $\mathfrak{g}= \slN =\mathfrak{sl}(N,\mathbb{C})$
be the Lie algebra of the
special linear algebraic group $G=\SLN =\SL(N,\mathbb{C})$. 
Here
$$\cE_{\slN} = \cE_{\glN}\Big/ \R \cdot (\ep{1} + \cdots + \ep{N})$$
and the weight lattice is 
$$\sLat = \sum_{i=1}^N \ZZ \epi \Big/
\ZZ \cdot (\ep{1} + \cdots + \ep{N})$$
which is a free $\ZZ$-module of rank $N-1$.
Let
$$\mathcal{E} = \mathcal{E}_{\slN} =
\{\e{i} \quad | \quad 1 \le i \le N\}$$
where $\e{i} = \epi + \ZZ \cdot (\ep{1} + \cdots + \ep{N})$.
As before these are the weights of $V = V_{\e{1}}$.
  Note $\e{1} + \cdots + \e{N}=0$.

$\cE_{\slN}$  has three bases we consider of which the first
two are also $\ZZ$-bases of $\sLat$.
The first basis is
$\{\e{i} \quad | \quad 1 \le i < N\} \subseteq \mathcal{E}$.
The second basis is $\{ \omega_i \mid 1\le i<N \}$
consisting of the fundamental weights $\omega_i =
\e{1} + \cdots + \e{i}$.
The weight lattice can also be expressed as
$$\sLat = \bigoplus_{i=1}^{N-1} \ZZ \omega_i
        = \bigoplus_{i=1}^{N-1} \ZZ \e{i}.$$
 The set of dominant integral weights is
\begin{gather*}
\sLatp =
 \sum_{i=1}^{N-1} \ZZ_{\ge 0}  \omega_i
=
\{ m_1\e{1}+ \cdots + m_{N-1}\e{N-1}\,\, \mid \,\,
m_i\in\mathbb{Z},\, m_1\geq \cdots \geq m_{N-1} \ge 0\} .
\end{gather*}
Our third basis is the set of simple roots
$$\Pi = \Pi^{A_{N-1}}=\{\alpha_{i} \mid 1 \le i < N \}$$
where $\alpha_i = \e{i} - \e{i+1}$.
The root lattice is $\Q =\Q^{A_{N-1}} = \bigoplus_{i=1}^{N-1} \ZZ \alpha_i$
and it is a sublattice of index $N = [\sLat : \Q^{A_{N-1}} ]$.
The set of roots, and the set of positive roots, are 
$$\Delta = \Delta^{A_{N-1}} = \{ \e{i}-\e{j} \mid 1 \le i , j \le N, i\neq j \},\qquad \Delta_+ = \Delta^{A_{N-1}}_+ = \{ \e{i}-\e{j} \mid 1 \le i < j \le N \}.$$
This is the root system of type $A_{N-1}$.

We have a symmetric form $\langle \, , \, \rangle$
on $\cE_{\slN}$ for which $[ \langle \alpha_i, \alpha_j \rangle ]_{i,j =1}^{N-1}$
yields the Cartan matrix of type $A_{N-1}$.
It is useful to note
\begin{equation}
\label{eq-form-on-e}
\langle \e{i}, \e{j} \rangle = \delta_{ij} - \frac{1}{N}.
\end{equation}

We have $\rho  \in \sLatp$ given by
$$ \rho = \frac 12 \sum_{\beta \in \Delta_+} \beta
= \frac 12 \left( (N-1)\e{1} + (N-3)\e{2} + (N-5)\e{3} + \cdots + (1-N) \e{N}
\right)
= \sum_{i=1}^{N-1} \omega_i.$$
We note
\begin{equation}
\label{eq-form-on-rho}
\langle \e{i}, 2\rho \rangle = N - (2i-1).
\end{equation}

The Young diagram $\YDs(\lambda)$ is defined the same way as for $GL$, and the only difference is in the diagonal labelling.  We say the size of $\lambda\in\sLatp$ is
$$|\lambda| = \sum_{i=1}^{N} (m_i-m_N)$$
which corresponds to the number of boxes in $\YDs(\lambda)$.
We label the principal diagonal $-\frac {|\lambda|}{N}$.  Hence the diagonal labels will lie in $-\frac pN + \ZZ$ if $\lambda$ is in the coset $\omega_p + Q$.

Note  that while $|\lambda|$ is well-defined, i.e. $|\lambda|=|\lambda+\wdet|$, the sum $\sum_i m_i = |\lambda| + N m_N$ is not.  Another natural choice of representative for $\lambda$ would be to choose the (possibly fractional) representative with $\sum_i m_i = 0,$ as follows
$$\lambda = (m_1 - \frac{\sum_i m_i}{N}, \ldots, m_{N-1} - \frac{\sum_i m_i}{N}, m_N - \frac{\sum_i m_i}{N}).$$
One may consider its upper left corner to lie in row 1, but column $1 + (m_N - \frac{\sum_i m_i}{N}) = 1-\frac{|\lambda|}{N}$.  This explains in part the appearance of fractional diagonal labels.

Another useful calculation is
\begin{equation}
\label{eq-form-on-lambda}
\langle \e{j}, \lambda \rangle =  m_j - \frac {\sum_i m_i}{N}=
  m_j -m_N - \frac 1N |\lambda|.
\end{equation}

\subsection{Walks on the weight lattice}

For the following definition we may take $G =\GLN, \SLN$.  Correspondingly, we let $\gsLatp=\gLatp$ or $\sLatp$.
Since $\ep{i}$ is a representative of $\e{i}$, we will use $\ep{i}$ below
to denote either one.

\begin{definition} A \emph{walk in $\gsLatp$ of length $n$}, from weight $\lambda$ to weight $\mu$ is a finite sequence,
$$\underline{u}=(\lambda = u_0, u_1,\ldots, u_n=\mu),$$
where each $u_i\in\gsLatp$, and each difference $u_{i}-u_{i-1}$ lies in $\mathcal{E}$.  We denote by $\delta_i(\underline{u})$ the index of $u_i-u_{i-1}\in\mathcal{E}$, so that $u_i-u_{i-1} = \ep{\delta_i(\underline{u})}$.
\end{definition}

\begin{remark} The condition that each $u_i\in\gsLatp$ means that we could call these ``dominant walks" or ``walks in the dominant chamber."  Since these are the only kinds of walks we will consider, we drop mention of the adjective dominant henceforth.
\end{remark}

\begin{definition}
A walk in $\sLatp$ which begins and ends at the same $\lambda$ is called a \emph{looped walk} at $\lambda$.
A walk in $\gLatp$ which begins at $\lambda$  and ends at $\lambda +k \wdet$
is also called a \emph{looped walk} at $\lambda$.  We denote by $\LWalk Nk{\lambda}$ the set of all looped walks at $\lambda$ of length $n=Nk$.
\end{definition}

See Figures \ref{fig-4stepwalk}, \ref{fig-walks-sl2}, and \ref{fig-walks-sl3} for examples of looped walks.

\begin{remark}
Since $\wdet = 0 \in \sLatp$, we can say in either case that a
looped walk begins at $\lambda$ and ends at $\lambda+k\wdet$.
\end{remark}

\begin{remark}Note that the multiset $\{ \ep{\delta_i(\uu)} \mid 1 \le i \le n
\}$ of steps taken on any looped walk $\uu$ consists of $\calE$ with integer multiplicity $k=n/N$.
\end{remark}

\subsection{Skew tableaux from walks on the weight lattice}
\label{sec-walks-to-skew}

For the rest of the paper, unless otherwise noted, we let $k \in \ZZ_{>0}$ and $n=Nk$.

We shall now recall an alternative combinatorial description of the set $\LWalk Nk{\lambda}$ of looped walks at $\lambda$ in terms of skew tableaux.  To begin, we associate to a weight $\lambda\in\gsLatp$ a skew diagram $\Dlam$ constructed in either of the following clearly equivalent ways
(See Figure \ref{fig-dual-lambda}):
\begin{itemize}
\item as the skew diagram $\Dlam = (\YDgs(\lambda) +(k^N))/\YDgs(\lambda)$.
\item as the skew diagram $\Dlam$ obtained by removing $\YDgs(\lambda)$ from the upper left, and $\YDgs(\lambda^*)$, rotated 180 degrees, from the lower right, of the $N\times(k+m_1-m_N)$ rectangular diagram.
\end{itemize}

The skew diagram $\Dlam$ inherits diagonal labels from $\YDgs(\lambda)$ as well as choice of principal diagonal.

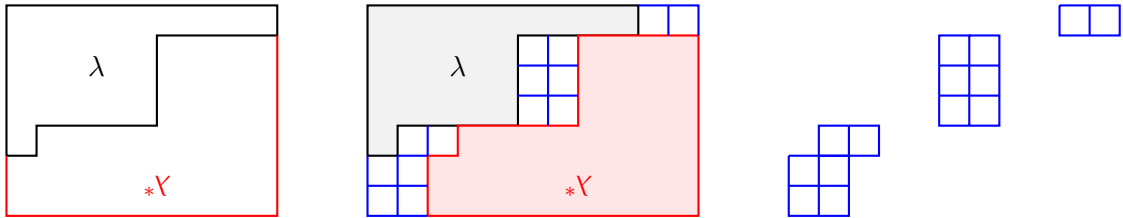
\begin{figure}[h]
\begin{tikzpicture}[scale=0.4]
\begin{scope}[shift={(0,0)}]
\draw [thick] (0,2) -- (1,2) -- (1,3) -- (5,3) -- (5,6) 
        -- (9,6) -- (9,7) -- (0,7) -- (0,2);
\draw [thick,red] (0,2) -- (0,0) --  (9,0) -- (9,6);
\node at (3,5) {$\lambda$};
\node [rotate=180,red] at (5,1) {$\lambda^*$};
\end{scope}
\begin{scope}[shift={(12,0)}]
\draw [thick,fill=gray!10] (0,2) -- (1,2) -- (1,3) -- (5,3) -- (5,6) 
        -- (9,6) -- (9,7) -- (0,7) -- (0,2);
\draw [thick,red,fill=red!10] (2,0) -- (2,2) -- (3,2) -- (3,3)
        -- (7,3) -- (7,6) -- (11,6) -- (11,0) -- (2,0);
\draw [thick,blue]  (0,2) -- (0,0) -- (2,0);
\draw [thick,blue]  (0,1) -- (2,1);
\draw [thick,blue]  (1,2) -- (2,2);
\draw [thick,blue]  (5,4) -- (7,4);
\draw [thick,blue]  (5,5) -- (7,5);
\draw [thick,blue]  (9,7) -- (11,7) --(11,6);
\draw [thick,blue]  (1,0) -- (1,2);
\draw [thick,blue]  (2,2) -- (2,3);
\draw [thick,blue]  (6,3) -- (6,6);
\draw [thick,blue]  (10,6) -- (10,7);
\node at (3,5) {$\lambda$};
\node [rotate=180,red] at (7,1) {$\lambda^*$};
\end{scope}

\begin{scope}[shift={(26,0)}]
\draw [thick,blue]  (0,2) -- (0,0) -- (2,0) -- (2,2) -- (0,2);
\draw [thick,blue]  (0,1) -- (2,1);
\draw [thick,blue]  (1,2) -- (2,2);
\draw [thick,blue]  (5,4) -- (7,4);
\draw [thick,blue]  (5,5) -- (7,5);
\draw [thick,blue]  (9,7) -- (11,7) --(11,6) -- (9,6) -- (9,7);
\draw [thick,blue]  (1,0) -- (1,3) -- (3,3) -- (3,2) -- (2,2);
\draw [thick,blue]  (2,2) -- (2,3);
\draw [thick,blue]  (6,3) -- (6,6) -- (5,6) -- (5,3) -- (7,3) -- (7,6) -- (6,6);
\draw [thick,blue]  (10,6) -- (10,7);
\end{scope}
\end{tikzpicture}
\caption{The skew diagram $\DNk 72{\lambda}$ in the case $N=7$, $k=2$, $n=14$.}
\label{fig-dual-lambda}
\end{figure}

 Recall for a (skew) diagram with $n$ boxes that a standard tableau is a filling of its boxes with $\{1, 2, \ldots, n\}$ such that entries increase across rows and down columns.

\begin{definition}\label{def-Skk}
 Given a weight $\lambda\in\gsLatp$, we denote by $\Skk$ the set of all standard tableaux of diagonal-labeled skew shape $\Dlam$.
\end{definition}

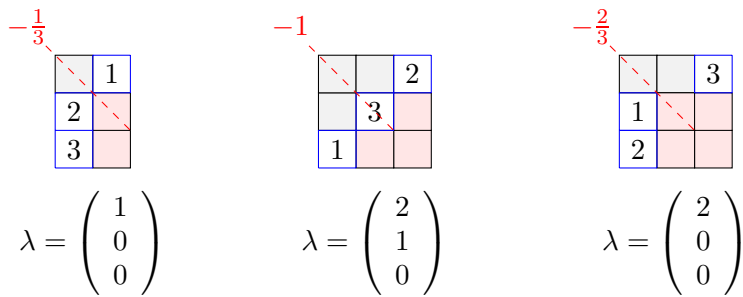
\begin{figure}[h]
\begin{tikzpicture}[scale=0.5]
\begin{scope}[shift={(0,0.5)}]
\draw [thin,fill=gray!10] (0,1) rectangle (1,2);
\draw [thin,fill=red!10] (1,-1) rectangle (2,0);
\draw [thin,fill=red!10] (1,0) rectangle (2,1);
\draw[step=1,blue] (0,-1) grid (1,1);
\draw[step=1,blue] (1,1) grid (2,2);
\draw (1.5,1.5) node {$1$};
\draw (0.5,0.5) node {$2$};
\draw (0.5,-0.5) node {$3$};
\draw [dashed,red] (-0.25,2.25) -- (2,0);
\draw (-0.75,2.75) node[red]  {$-\frac 13$};
\draw (1,-3) node {$\lambda=\left(\begin{array}{c} 1\\0\\0 \end{array}\right)$};
\end{scope}

\begin{scope}[shift={(8,0.5)}]
\draw [thin,fill=gray!10] (-1,0) rectangle (0,1);
\draw [thin,fill=gray!10] (-1,1) rectangle (0,2);
\draw [thin,fill=gray!10] (0,1) rectangle (1,2);
\draw [thin,fill=red!10] (1,-1) rectangle (2,0);
\draw [thin,fill=red!10] (1,0) rectangle (2,1);
\draw [thin,fill=red!10] (0,0) rectangle (1,-1);
\draw[step=1,blue] (0,0) grid (1,1);
\draw[step=1,blue] (1,1) grid (2,2);
\draw[step=1,blue] (-1,-1) grid (0,0);
\draw (1.5,1.5) node {$2$};
\draw (0.5,0.5) node {$3$};
\draw (-0.5,-0.5) node {$1$};
\draw [dashed,red] (-1.25,2.25) -- (1,0);
\draw (-1.75,2.75) node[red]  {$-1$};
\draw (.5,-3) node {$\lambda=\left(\begin{array}{c} 2\\1\\0 \end{array}\right)$};
\end{scope}

\begin{scope}[shift={(16,0.5)}]
\draw [thin,fill=gray!10] (-1,1) rectangle (0,2);
\draw [thin,fill=gray!10] (0,1) rectangle (1,2);
\draw [thin,fill=red!10] (0,0) rectangle (1,1);
\draw [thin,fill=red!10] (1,-1) rectangle (2,0);
\draw [thin,fill=red!10] (1,0) rectangle (2,1);
\draw [thin,fill=red!10] (0,0) rectangle (1,-1);
\draw[step=1,blue] (-1,0) grid (0,1);
\draw[step=1,blue] (1,1) grid (2,2);
\draw[step=1,blue] (-1,-1) grid (0,0);
\draw (1.5,1.5) node {$3$};
\draw (-0.5,0.5) node {$1$};
\draw (-0.5,-0.5) node {$2$};
\draw [dashed,red] (-1.25,2.25) -- (1,0);
\draw (-1.75,2.75) node[red]  {$-\frac 23$};
\draw (.5,-3) node {$\lambda=\left(\begin{array}{c} 2\\0\\0 \end{array}\right)$};
\end{scope}
\end{tikzpicture}
\caption{$G=\SL_3$, $k=1$. We list $\cT \in \Skk$ with principal
diagonal labeled $-\frac{|\lambda|}{3}$. Once periodized
via $\SPer(\cT)$, these
tableaux form a $\pi$ orbit, see Figure 
\ref{fig-sl-diagonals-periodic}.}
\label{fig-sl3-diagonals-skew}
\end{figure}

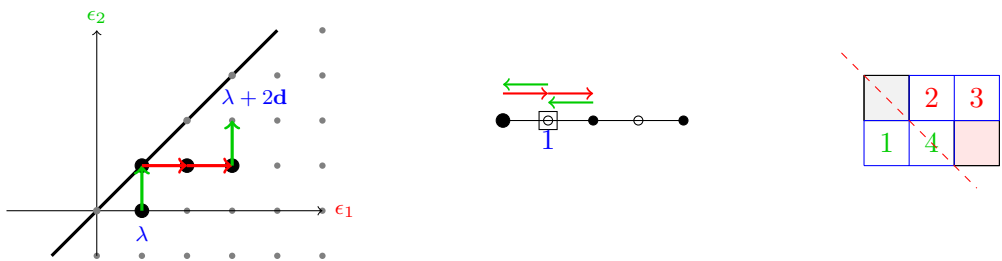
\begin{figure}[h]
\begin{tikzpicture}[scale=0.6]
\draw[very thick] (-1,-1)--(4,4);
\begin{scope}[shift={(0,0)}]
 \foreach \i in {-2,...,5}
      \foreach \j in {-1,...,4}{
        \ifnum \i > \j -1
            \fill[gray] (\i,\j) circle(2pt);
        \fi
      };
\draw[->] (-2,0)--(5,0);
\draw[->] (0,-1)--(0,4);
\draw [fill]  (1,0) circle (.15cm);
\draw [fill]  (1,1) circle (.15cm);
\draw [fill]  (2,1) circle (.15cm);
\draw [fill]  (3,1) circle (.15cm);
\draw [fill,gray]  (2,2) circle (2pt);
\draw [fill,gray]  (0,0) circle (2pt);
\draw [fill,gray]  (3,3) circle (2pt);

\draw (1.0,-.5) node[blue] {\scriptsize $\lambda$};
\draw (3.5,2.5) node[blue] {\scriptsize $\lambda  + 2\wdet$};
\draw (5.5,0.0) node[red] {\scriptsize $\ep{1}$};
\draw (0.0,4.3) node[green!80!black] {\scriptsize $\ep{2}$};
\draw[->,very thick,  green!80!black] (1,0)--(1,1);
\draw[->, very thick, red] (1,1)--(2,1);
\draw[->, very thick, red] (2,1)--(3,1);
\draw[->,very thick,  green!80!black] (3,1)--(3,2);
\end{scope}

\begin{scope}[shift={(7,2)}]
\draw [fill] (2,0) circle (.15cm);
\draw [fill] (2,0) circle (.1cm);
\draw [fill] (4,0) circle (.1cm);
\draw [fill] (6,0) circle (.1cm);
\draw  (2.8,-0.2) rectangle (3.2,0.2);
\draw [thick,->,green!80!black] (3,0.8) -- (2,0.8);
\draw [thick,->,red] (2,0.6) -- (3,0.6);
\draw [thick,->,red] (3,0.6) -- (4,0.6);
\draw [thick,->,green!80!black] (4,0.4) -- (3,0.4);
\draw  (3,0) circle (.1cm);
\draw  (5,0) circle (.1cm);
\node[below]  at (3,0) {\color{blue} $1$};
\draw (2,0)--(3,0)--(4,0)--(5,0)--(6,0);
\end{scope}


\begin{scope}[shift={(17,1)}]
\draw [thin,fill=gray!10] (0,1) rectangle (1,2);
\draw [thin,fill=red!10] (2,0) rectangle (3,1);
\draw[step=1,black] (0,1) grid (1,2);
\draw[step=1,black] (2,0) grid (3,1);
\draw[step=1,blue] (0,0) grid (2,1);
\draw[step=1,blue] (1,1) grid (3,2);
\draw (1.5,1.5) node[red] {$2$};
\draw (2.5,1.5) node[red] {$3$};
\draw (0.5,0.5) node[green!80!black] {$1$};
\draw (1.5,0.5) node[green!80!black] {$4$};
\draw [dashed,red] (-0.5,2.5) -- (2.5, -0.5);
\end{scope}

\end{tikzpicture}
\caption{
A looped walk $\uu$ at $\lambda = \ep{1}\in \gLatp$ with $4$ steps, its projection in $\sLatp$ for $N=2$, and the skew tableau $\Tab(\uu)$ with the principal diagonal indicated by a dashed line.  For $\gLatp$ the principal diagonal is labelled $0$, while for $\sLatp$ it is labelled $-\frac{1}{2}$.
}\label{fig-4stepwalk}
\end{figure}

 \begin{definition}\label{def:Tab}
Define the map $\Tab:\LWalk Nk{\lambda} \to
\Sk Nk{\lambda},$
from length $n =Nk$  looped walks at $\lambda \in \gsLatp$ 
to standard skew tableaux of shape $\Dlam$
as follows:  for each $i=1,\ldots, n$
fill the leftmost vacant box in the $\delta_i(\uu)$-th row of $\Dlam$
with the symbol $i$.
\end{definition}

We note that the map $\Tab$ doesn't reference the extra data of the diagonal labelling.  When it is important to differentiate the $SL$ setting, we will use $\STab(\lambda)$ for $\lambda \in \sLatp$.  The following is essentially proved in \cite{OR}:

\begin{proposition}
The map $\Tab:\LWalk Nk{\lambda} \xrightarrow{\sim} \Sk Nk{\lambda}$
is a bijection.
\end{proposition}

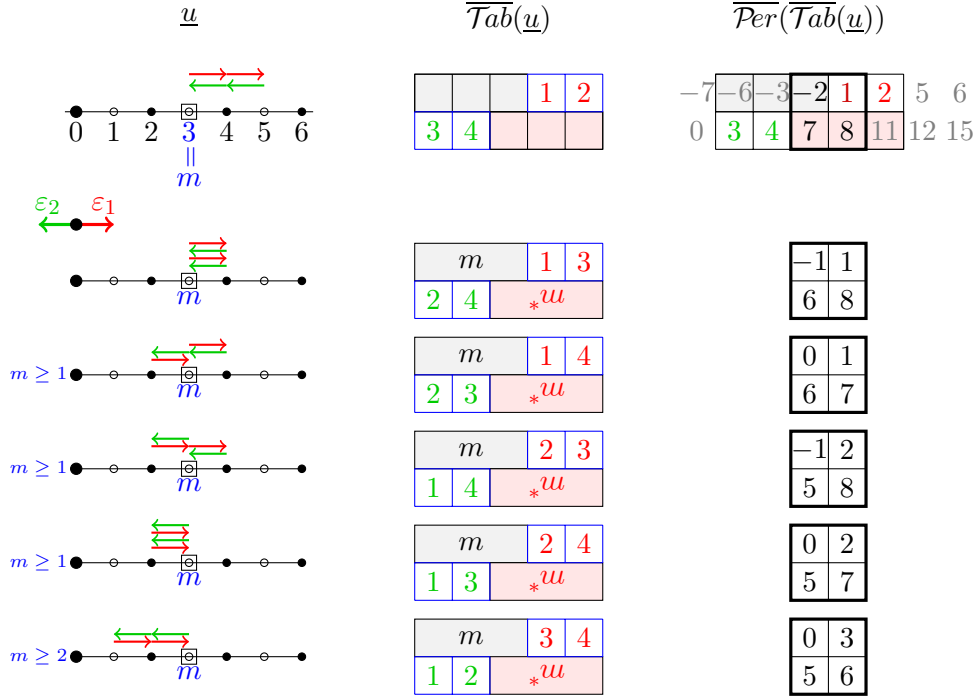
\begin{figure}[h]
\begin{tikzpicture}[scale=0.5]
\node at (3,13) {$\uu$};
\node at (11.5,13) {$\STab(\uu)$};
\node at (19.5,13) {$\SPer(\STab(\uu))$};

\begin{scope}[shift={(0,7.5)}]

\draw [very thick,->,red] (0,0) -- (1,0);
\draw [very thick,->,green!80!black] (0,0) -- (-1,0);
\draw [fill] (0,0) circle (.15cm);
\node at (.75,.5) {\color{red} $\e{1}$};
\node at (-.75,.5) {\color{green!80!black} $\e{2}$};

\end{scope}
\begin{scope}[shift={(0,10.5)}]
\draw [fill] (0,0) circle (.15cm);
\draw [fill] (2,0) circle (.1cm);
\draw [fill] (4,0) circle (.1cm);
\draw [fill] (6,0) circle (.1cm);
\draw  (1,0) circle (.1cm);
\draw  (3,0) circle (.1cm);
\draw  (5,0) circle (.1cm);
\draw  (2.8,-0.2) rectangle (3.2,0.2);
\draw [thick,->,red] (3,1) -- (4,1);
\draw [thick,->,red] (4,1) -- (5,1);
\draw [thick,->,green!80!black] (5,0.7) -- (4,0.7);
\draw [thick,->,green!80!black] (4,0.7) -- (3,0.7);
\node[below] at (0,0) {$0$};
\node[below] at (1,0)  {$1$};
\node[below] at (2,0) {$2$};
\node[below]  at (3,0) {\color{blue} $3$};
\node[below]  at (4,0) {$4$};
\node[below]  at (5,0) {$5$};
\node[below]  at (6,0) {$6$};
\node [rotate=90] at (3.1,-1.2) {\color{blue} $=$};
\node at (3,-1.8) {\color{blue} $m$};

\draw (0,0)--(1,0)--(2,0)--(3,0)--(4,0)--(5,0)--(6,0);
\draw (-0.3,0)--(0,0) (6,0)--(6.3,0);  
\end{scope}

\begin{scope}[shift={(9,9.5)}]
 \draw [thin,fill=gray!10] (0,1) rectangle (3,2);
\draw [thin,fill=red!10] (2,0) rectangle (5,1);
\draw[step=1,black] (0,1) grid (3,2);
\draw[step=1,black] (2,0) grid (5,1);
\draw[step=1,blue] (0,0) grid (2,1);
\draw[step=1,blue] (3,1) grid (5,2);
\draw (3.5,1.5) node[red] {$1$};
\draw (4.5,1.5) node[red] {$2$};
\draw (0.5,0.5) node[green!80!black] {$3$};
\draw (1.5,0.5) node[green!80!black] {$4$};
\end{scope}

\begin{scope}[shift={(17,9.5)}]
\draw [thin,fill=gray!10] (0,1) rectangle (3,2);
\draw [thin,fill=red!10] (2,0) rectangle (5,1);
\draw [very thick] (2,0) rectangle (4,2);
\draw[step=1] (0,0) grid (5,2);
\draw (0.5,1.5) node[gray]  {$-6$};
\draw (1.5,1.5) node[gray]  {$-3$};
\draw (2.5,1.5) node[black]  {$-2$};
\draw (3.5,1.5) node[red!80!black] {$1$};
\draw (4.5,1.5) node[red] {$2$};
\draw (0.5,0.5) node[green!80!black] {$3$};
\draw (1.5,0.5) node[green!80!black] {$4$};
\draw (2.5,0.5) node[black]  {$7$};
\draw (3.5,0.5) node[black]  {$8$};
\draw (4.5,0.5) node[gray]  {$11$};

\draw (5.5,1.5) node[gray] {$5$};
\draw (6.5,0.5) node[gray]  {$15$};
\draw (5.5,0.5) node[gray] {$12$};
\draw (6.5,1.5) node[gray]  {$6$};
\draw (-.5,1.5) node[gray]  {$-7$};
\draw (-.5,0.5) node[gray]  {$0$};
\end{scope}

\begin{scope}[shift={(0,6)}]
\draw [fill] (0,0) circle (.15cm);
\draw [fill] (2,0) circle (.1cm);
\draw [fill] (4,0) circle (.1cm);
\draw [fill] (6,0) circle (.1cm);
\draw  (2.8,-0.2) rectangle (3.2,0.2);
\draw [thick,->,red] (3,1) -- (4,1);
\draw [thick,->,green!80!black] (4,0.8) -- (3,0.8);
\draw [thick,->,red] (3,0.6) -- (4,0.6);
\draw [thick,->,green!80!black] (4,0.4) -- (3,0.4);
\draw  (1,0) circle (.1cm);
\draw  (3,0) circle (.1cm);
\draw  (5,0) circle (.1cm);
\node[below]  at (3,0) {\color{blue} $m$};

\draw (0,0)--(1,0)--(2,0)--(3,0)--(4,0)--(5,0)--(6,0);
\end{scope}

\begin{scope}[shift={(9,5)}][scale = 0.5]
\draw [thin,fill=gray!10] (0,1) rectangle (3,2);
\draw [thin,fill=red!10] (2,0) rectangle (5,1);
\draw[step=1,blue] (0,0) grid (2,1);
\draw[step=1,blue] (3,1) grid (5,2);
\draw (3.5,1.5) node[red] {$1$};
\draw (4.5,1.5) node[red] {$3$};
\draw (0.5,0.5) node[green!80!black] {$2$};
\draw (1.5,0.5) node[green!80!black] {$4$};
\node at (1.5,1.5) {$m$};
\node [rotate=180,red] at (3.5,.5) {$m^*$};
\end{scope}

\begin{scope}[shift={(17,5)}]
\draw [very thick] (2,0) rectangle (4,2);
\draw[step=1] (2,0) grid (4,2);
\draw (2.5,1.5) node[black]  {$-1$};
\draw (3.5,1.5) node[black] {$1$};
\draw (2.5,0.5) node[black]  {$6$};
\draw (3.5,0.5) node[black]  {$8$};
\end{scope}

\begin{scope}[shift={(0,3.5)}]
\draw [fill] (0,0) circle (.15cm);
\draw [fill] (2,0) circle (.1cm);
\draw [fill] (4,0) circle (.1cm);
\draw [fill] (6,0) circle (.1cm);
\draw  (2.8,-0.2) rectangle (3.2,0.2);
\draw [thick,->,red] (3,0.8) -- (4,0.8);
\draw [thick,->,green!80!black] (4,0.6) -- (3,0.6);
\draw [thick,->,green!80!black] (3,0.6) -- (2,0.6);
\draw [thick,->,red] (2,0.4) -- (3,0.4);
\draw  (1,0) circle (.1cm);
\draw  (3,0) circle (.1cm);
\draw  (5,0) circle (.1cm);
\node[below]  at (3,0) {\color{blue} $m$};
\node  at (-1,0) {\tiny{\color{blue} $m \ge 1$}};

\draw (0,0)--(1,0)--(2,0)--(3,0)--(4,0)--(5,0)--(6,0);
\end{scope}

\begin{scope}[shift={(9,2.5)}]
\draw [thin,fill=gray!10] (0,1) rectangle (3,2);
\draw [thin,fill=red!10] (2,0) rectangle (5,1);
\draw[step=1,blue] (0,0) grid (2,1);
\draw[step=1,blue] (3,1) grid (5,2);
\draw (3.5,1.5) node[red] {$1$};
\draw (4.5,1.5) node[red] {$4$};
\draw (0.5,0.5) node[green!80!black] {$2$};
\draw (1.5,0.5) node[green!80!black] {$3$};
\node at (1.5,1.5) {$m$};
\node [rotate=180,red] at (3.5,.5) {$m^*$};
\end{scope}


\begin{scope}[shift={(17,2.5)}]
\draw [very thick] (2,0) rectangle (4,2);
\draw[step=1] (2,0) grid (4,2);
\draw (2.5,1.5) node[black]  {$0$};
\draw (3.5,1.5) node[black] {$1$};
\draw (2.5,0.5) node[black]  {$6$};
\draw (3.5,0.5) node[black]  {$7$};
\end{scope}

\begin{scope}[shift={(0,1)}]
\draw [fill] (0,0) circle (.15cm);
\draw [fill] (2,0) circle (.1cm);
\draw [fill] (4,0) circle (.1cm);
\draw [fill] (6,0) circle (.1cm);
\draw  (2.8,-0.2) rectangle (3.2,0.2);
\draw [thick,->,green!80!black] (3,0.8) -- (2,0.8);
\draw [thick,->,red] (2,0.6) -- (3,0.6);
\draw [thick,->,red] (3,0.6) -- (4,0.6);
\draw [thick,->,green!80!black] (4,0.4) -- (3,0.4);
\draw  (1,0) circle (.1cm);
\draw  (3,0) circle (.1cm);
\draw  (5,0) circle (.1cm);
\node[below]  at (3,0) {\color{blue} $m$};
\node  at (-1,0) {\tiny{\color{blue} $m \ge 1$}};
\draw (0,0)--(1,0)--(2,0)--(3,0)--(4,0)--(5,0)--(6,0);
\end{scope}

\begin{scope}[shift={(9,0)}]
\draw [thin,fill=gray!10] (0,1) rectangle (3,2);
\draw [thin,fill=red!10] (2,0) rectangle (5,1);
\draw[step=1,blue] (0,0) grid (2,1);
\draw[step=1,blue] (3,1) grid (5,2);
\draw (3.5,1.5) node[red] {$2$};
\draw (4.5,1.5) node[red] {$3$};
\draw (0.5,0.5) node[green!80!black] {$1$};
\draw (1.5,0.5) node[green!80!black] {$4$};
\node at (1.5,1.5) {$m$};
\node [rotate=180,red] at (3.5,.5) {$m^*$};
\end{scope}

\begin{scope}[shift={(17,0)}]
\draw [very thick] (2,0) rectangle (4,2);
\draw[step=1] (2,0) grid (4,2);
\draw (2.5,1.5) node[black]  {$-1$};
\draw (3.5,1.5) node[black] {$2$};
\draw (2.5,0.5) node[black]  {$5$};
\draw (3.5,0.5) node[black]  {$8$};
\end{scope}

\begin{scope}[shift={(0,-1.5)}]
\draw [fill] (0,0) circle (.15cm);
\draw [fill] (2,0) circle (.1cm);
\draw [fill] (4,0) circle (.1cm);
\draw [fill] (6,0) circle (.1cm);
\draw  (2.8,-0.2) rectangle (3.2,0.2);
\draw [thick,->,green!80!black] (3,1) -- (2,1);
\draw [thick,->,red] (2,0.8) -- (3,0.8);
\draw [thick,->,green!80!black] (3,0.6) -- (2,0.6);
\draw [thick,->,red] (2,0.4) -- (3,0.4);
\draw  (1,0) circle (.1cm);
\draw  (3,0) circle (.1cm);
\draw  (5,0) circle (.1cm);
\node[below]  at (3,0) {\color{blue} $m$};
\node  at (-1,0) {\tiny{\color{blue} $m \ge 1$}};

\draw (0,0)--(1,0)--(2,0)--(3,0)--(4,0)--(5,0)--(6,0);
\end{scope}

\begin{scope}[shift={(9,-2.5)}]
\draw [thin,fill=gray!10] (0,1) rectangle (3,2);
\draw [thin,fill=red!10] (2,0) rectangle (5,1);
\draw[step=1,blue] (0,0) grid (2,1);
\draw[step=1,blue] (3,1) grid (5,2);
\draw (3.5,1.5) node[red] {$2$};
\draw (4.5,1.5) node[red] {$4$};
\draw (0.5,0.5) node[green!80!black] {$1$};
\draw (1.5,0.5) node[green!80!black] {$3$};
\node at (1.5,1.5) {$m$};
\node [rotate=180,red] at (3.5,.5) {$m^*$};
\end{scope}

\begin{scope}[shift={(17,-2.5)}]
\draw [very thick] (2,0) rectangle (4,2);
\draw[step=1] (2,0) grid (4,2);
\draw (2.5,1.5) node[black]  {$0$};
\draw (3.5,1.5) node[black] {$2$};
\draw (2.5,0.5) node[black]  {$5$};
\draw (3.5,0.5) node[black]  {$7$};
\end{scope}

\begin{scope}[shift={(0,-4)}]
\draw [fill] (0,0) circle (.15cm);
\draw [fill] (2,0) circle (.1cm);
\draw [fill] (4,0) circle (.1cm);
\draw [fill] (6,0) circle (.1cm);
\draw  (2.8,-0.2) rectangle (3.2,0.2);
\draw [thick,->,green!80!black] (3,0.6) -- (2,0.6);
\draw [thick,->,green!80!black] (2,0.6) -- (1,0.6);
\draw [thick,->,red] (1,0.4) -- (2,0.4);
\draw [thick,->,red] (2,0.4) -- (3,0.4);
\draw  (1,0) circle (.1cm);
\draw  (3,0) circle (.1cm);
\draw  (5,0) circle (.1cm);
\node[below]  at (3,0) {\color{blue} $m$};
\node  at (-1,0) {\tiny{\color{blue} $m \ge 2$}};

\draw (0,0)--(1,0)--(2,0)--(3,0)--(4,0)--(5,0)--(6,0);
\end{scope}

\begin{scope}[shift={(9,-5)}]
\draw [thin,fill=gray!10] (0,1) rectangle (3,2);
\draw [thin,fill=red!10] (2,0) rectangle (5,1);
\draw[step=1,blue] (0,0) grid (2,1);
\draw[step=1,blue] (3,1) grid (5,2);
\draw (3.5,1.5) node[red] {$3$};
\draw (4.5,1.5) node[red] {$4$};
\draw (0.5,0.5) node[green!80!black] {$1$};
\draw (1.5,0.5) node[green!80!black] {$2$};
\node at (1.5,1.5) {$m$};
\node [rotate=180,red] at (3.5,.5) {$m^*$};
\end{scope}

\begin{scope}[shift={(17,-5)}]
\draw [very thick] (2,0) rectangle (4,2);
\draw[step=1] (2,0) grid (4,2);
\draw (2.5,1.5) node[black]  {$0$};
\draw (3.5,1.5) node[black] {$3$};
\draw (2.5,0.5) node[black]  {$5$};
\draw (3.5,0.5) node[black]  {$6$};
\end{scope}

\end{tikzpicture}
\caption{The case $G=SL_2$, $n=4$. 
The first column lists the four-step walks 
in $\LWalk 22{m}$, of which there are
six when $m$ is large.  The second column lists the standard skew tableaux
$\Sk 22{m}$.  The final column lists fundamental domains of
the six corresponding standard periodic tableaux in 
$\psytn{2}(2^2)\big/ \pi^4$ (see Section \ref{sec-SLpsyt}), in the case $m=3$.
For purposes of illustration, we also take $m=3$ across the entire first row;
otherwise $m$ is free.  The placement of rows indicates the bijections $\STab$
of Definition \ref{def:Tab} and $\SPer$ of Definition \ref{def:psi}.}

\label{fig-walks-sl2}\end{figure}

\begin{figure}

\begin{tikzpicture}[scale=0.5]
\node at (1.5,14) {$\uu$};
\node at (9,14) {$\STab(\uu)$};
\node at (16.5,14) {$\SPer(\STab(\uu))$};
\begin{scope}[shift={(0,10)}]
\draw [fill] (0,0) circle (.15cm);
\draw  (1,0) circle (.09cm);
\draw  (2,0) circle (.09cm);
\draw [fill] (3,0) circle (.1cm);

\draw [fill] (1.5,.866)   circle (.09cm);
\draw  (0.5,.866)   circle (.09cm);
\draw  (2.5,.866)   circle (.09cm);

\draw (1,1.732) circle (.09cm);
\draw (2,1.732) circle (.09cm);

\draw [fill] (1.5,2.598)   circle (.09cm);

\draw  (0.8,-0.2) rectangle (1.2,0.2);
\draw (0,0)--(3.3,0) (0,0)--(1.5,2.598);
\draw [very thick,->,red] (1,0) -- (2,0);
\draw [very thick,->,green!80!black] (2,0) -- (1.5,.866);
\draw [very thick,->,blue] (1.5,.866) -- (1,0);
\node[below]  at (1,0) {\tiny {\color{blue} $\lambda = \omega_1$}};

\end{scope}

\begin{scope}[shift={(8,10.5)}]
\draw [thin,fill=gray!10] (0,1) rectangle (1,2);
\draw [thin,fill=red!10] (1,1) rectangle (2,-1);
\draw[step=1,blue] (1,1) grid (2,2);
\draw[step=1,blue] (0,1) grid (1,-1);
\draw (1.5,1.5) node[red] {$1$};
\draw (0.5,0.5) node[green!80!black] {$2$};
\draw (0.5,-0.5) node[blue] {$3$};
\node at (0.5,1.5) {$\lambda$};
\node [rotate=180,red] at (1.5,0) {$\lambda^*$};
\end{scope}


\begin{scope}[shift={(15,10.5)}]
\draw [thin,fill=gray!10] (0,1) rectangle (1,2);
\draw [thin,fill=red!10] (1,1) rectangle (2,-1);
\draw[step=1,black] (0,-1) grid (2,2);
\draw (-0.5,1.5) node[gray] {$-5$};
\draw (0.5,1.5) node[gray] {$-2$};
\draw (-1.5,1.5) node[gray] {$-8$};
\draw (-0.5,0.5) node[gray] {$-1$};
\draw (-1.5,0.5) node[gray] {$-4$};
\draw (-0.5,-0.5) node[gray] {$0$};
\draw (-1.5,-0.5) node[gray] {$-3$};
\draw (3.5,0.5) node[gray] {$11$};
\draw (2.5,0.5) node[gray] {$8$};
\draw (3.5,-0.5) node[gray] {$12$};
\draw (2.5,-0.5) node[gray] {$9$};
\draw (1.5,0.5) node[gray] {$5$};
\draw (1.5,-0.5) node[gray] {$6$};
\draw (2.5,1.5) node[gray] {$4$};
\draw (3.5,1.5) node[gray] {$7$};
\draw (1.5,1.5) node[red] {$1$};
\draw (0.5,0.5) node[green!80!black] {$2$};
\draw (0.5,-0.5) node[blue] {$3$};
\draw [very thick] (1,-1) rectangle (2,2);
\end{scope}

\begin{scope}[shift={(0,5)}]
\draw [fill] (0,0) circle (.15cm);
\draw  (1,0) circle (.09cm);
\draw  (2,0) circle (.09cm);
\draw [fill] (3,0) circle (.1cm);

\draw [fill] (1.5,.866)   circle (.09cm);
\draw  (0.5,.866)   circle (.09cm);
\draw  (2.5,.866)   circle (.09cm);
\draw (1,1.732) circle (.09cm);
\draw (2,1.732) circle (.09cm);
\draw [fill] (1.5,2.598)   circle (.09cm);

\draw  (0.8,-0.2) rectangle (1.2,0.2);
\draw (0,0)--(3.3,0) (0,0)--(1.5,2.598);
\draw [very thick,->,red] (.5,.866) -- (1.5,.866);
\draw [very thick,->,green!80!black] (1,0) -- (0.5,.866);
\draw [very thick,->,blue] (1.5,.866) -- (1,0);

\end{scope}

\begin{scope}[shift={(8,5.5)}]
\draw [thin,fill=gray!10] (0,1) rectangle (1,2);
\draw [thin,fill=red!10] (1,1) rectangle (2,-1);
\draw[step=1,blue] (1,1) grid (2,2);
\draw[step=1,blue] (0,1) grid (1,-1);
\draw (1.5,1.5) node[red] {$2$};
\draw (0.5,0.5) node[green!80!black] {$1$};
\draw (0.5,-0.5) node[blue] {$3$};
\node at (0.5,1.5) {$\lambda$};
\node [rotate=180,red] at (1.5,0) {$\lambda^*$};
\end{scope}

\begin{scope}[shift={(13,4.5)}]
\draw [very thick] (3,0) rectangle (4,3);
\draw[step=1] (3,0) grid (4,3);
\draw (3.5,2.5) node[black] {$2$};
\draw (3.5,1.5) node[black]  {$4$};
\draw (3.5,0.5) node[black]  {$6$};
\end{scope}

\begin{scope}[shift={(0,0)}]
\draw [fill] (0,0) circle (.15cm);
\draw  (1,0) circle (.09cm);
\draw  (2,0) circle (.09cm);
\draw [fill] (3,0) circle (.1cm);

\draw [fill] (1.5,.866)   circle (.09cm);
\draw  (0.5,.866)   circle (.09cm);
\draw  (2.5,.866)   circle (.09cm);
\draw (1,1.732) circle (.09cm);
\draw (2,1.732) circle (.09cm);
\draw [fill] (1.5,2.598)   circle (.09cm);

\draw  (0.8,-0.2) rectangle (1.2,0.2);
\draw (0,0)--(3.3,0) (0,0)--(1.5,2.598);
\draw [very thick,->,red] (0,0) -- (1,0);
\draw [very thick,->,green!80!black] (1,0) -- (0.5,.866);
\draw [very thick,->,blue] (0.5,.866) -- (0,0);
\end{scope}

\begin{scope}[shift={(8,0.5)}]
\draw [thin,fill=gray!10] (0,1) rectangle (1,2);
\draw [thin,fill=red!10] (1,1) rectangle (2,-1);
\draw[step=1,blue] (1,1) grid (2,2);
\draw[step=1,blue] (0,1) grid (1,-1);
\draw (1.5,1.5) node[red] {$3$};
\draw (0.5,0.5) node[green!80!black] {$1$};
\draw (0.5,-0.5) node[blue] {$2$};
\node at (0.5,1.5) {$\lambda$};
\node [rotate=180,red] at (1.5,0) {$\lambda^*$};
\end{scope}

\begin{scope}[shift={(13,-0.5)}]
\draw [very thick] (3,0) rectangle (4,3);
\draw[step=1] (3,0) grid (4,3);

\draw (3.5,2.5) node[black] {$3$};
\draw (3.5,1.5) node[black]  {$4$};
\draw (3.5,0.5) node[black]  {$5$};
\end{scope}

\begin{scope}[shift={(5,8)}]
\draw [very thick,->,red] (0,0) -- (1,0);
\draw [very thick,->,green!80!black] (0,0) -- (-0.5,.866);
\draw [very thick,->,blue] (0,0) -- (-0.5,-.866);
\draw [fill] (0,0) circle (.15cm);
\node[below]  at (1,0) {\color{red} $\e{1}$};
\node[left]  at (-0.5,.866) {\color{green!80!black} $\e{2}$};
\node[below]  at (-0.5,-.866) {\color{blue} $\e{3}$};
\end{scope}

\end{tikzpicture}
\caption{
The set $\LWalk 31{\omega_{1}}$ of three-step looped walks
in $\sLatp$ at $\omega_1$.  The allowed steps $\e{1}$, $\e{2}$ and $\e{3}$ are
shown to the right of the first column.
As in Figure \ref{fig-walks-sl2} the second and third columns list the
corresponding tableaux in $\Sk 31{\omega_1}$ and $\psytn{3}(3^1) \big/ \pi^3$
under $\STab$ from Definition \ref{def:Tab} and $\SPer$ from Definition \ref{def:psi}. }
\label{fig-walks-sl3}
\end{figure}
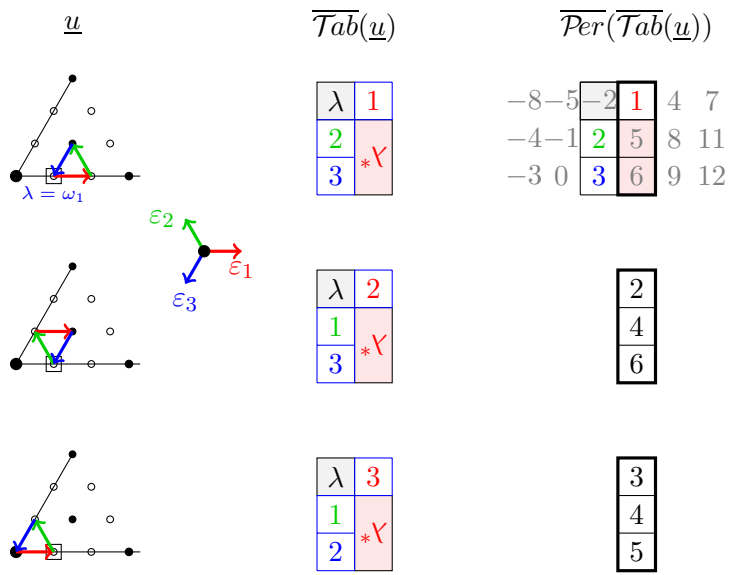

\begin{example}
\label{ex-walk-skew-sl2}
For the first walk in Figure  \ref{fig-walks-sl2},
$$\uu = (m, m+{\color{red} \e 1}, m+\e 1 + {\color{red} \e 1},
m+\e 1 + \e 1 + {\color{green!80!black}\e{2}} ,
m+\e 1 + \e 1 + \e{2} +{\color{green!80!black} \e{2}} = m),$$
and so the sequence
$$(\delta_1(\uu), \delta_2(\uu), \delta_3(\uu), \delta_4(\uu)) 
= ({\color{red}1},{\color{red}  1}, {\color{green!80!black} 2}, {\color{green!80!black} 2}).$$
Compare this to the first skew tableau
$\T = \Tab(\uu)$
in Figure  \ref{fig-walks-sl2}
which places 1 and 2 in the first row, 3 and 4 in the second row. 

See Figure \ref{fig-exit-dominant} to see how leaving the dominant chamber results in the tableau becoming non-standard.
In particular, if $m>1$, then the walk would not leave the chamber
and the ${\color{red} 3}$ would not be directly above the
${\color{green!80!black} 2}$.

\end{example}

\begin{figure}[h]
\begin{tikzpicture}[scale=0.5]
\begin{scope}[shift={(0,0)}]
\draw [fill] (2,0) circle (.2cm);
\draw [fill] (4,0) circle (.1cm);
\draw [fill] (6,0) circle (.1cm);
\draw  (2.8,-0.2) rectangle (3.2,0.2);
\draw [thick,->,green!80!black] (3,0.6) -- (2,0.6);
\draw [thick,->,green!80!black] (2,0.6) -- (1,0.6);
\draw [thick,->,red] (1,0.4) -- (2,0.4);
\draw [thick,->,red] (2,0.4) -- (3,0.4);
\draw  (3,0) circle (.1cm);
\draw  (5,0) circle (.1cm);
\node [rotate=90] at (3.1,-1.2) {\color{blue} $=$};
\node at (3,-1.8) {\color{blue} $m$};
\node[below] at (2,0) {$0$};
\node[below]  at (3,0) {\color{blue} $1$};

\draw (2,0)--(3,0)--(4,0)--(5,0)--(6,0);
\end{scope}

\begin{scope}[shift={(9,-1)}]
\draw [thin,fill=gray!10] (0,1) rectangle (1,2);
\draw [thin,fill=red!10] (2,0) rectangle (3,1);
\draw[step=1,black] (0,1) grid (1,2);
\draw[step=1,black] (2,0) grid (3,1);
\draw[step=1,blue] (0,0) grid (2,1);
\draw[step=1,blue] (1,1) grid (3,2);
\draw (1.5,1.5) node[red] {$3$};
\draw (2.5,1.5) node[red] {$4$};
\draw (0.5,0.5) node[green!80!black] {$1$};
\draw (1.5,0.5) node[green!80!black] {$2$};
\end{scope}
\end{tikzpicture}
\caption{The above is not a walk as it exits the dominant chamber when $m=1$; likewise the
corresponding skew tableau assigned by $\STab$ is not standard.}
\label{fig-exit-dominant}
\end{figure}
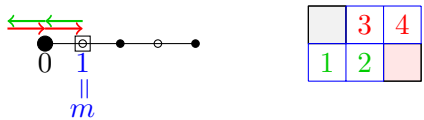

\section{Quantum algebras: $U_\q(\g)$, $\cO_\q(G)$ and $\cD_\q(G)$}
\label{sec-quantum}

\subsection{The quantum groups $U_\q(\mathfrak{gl}_N)$, $U_\q(\mathfrak{sl}_N)$}
We will consider the braided tensor categories $\Rep_\q(\GL_N)$ and $\Rep_\q(SL_N)$ of integrable $\cU_{\q}(\mathfrak{gl}_N)$-modules (resp. $\cU_{\q}(\mathfrak{sl}_N)$-modules).  We refer to \cite{KlSch} for detailed definitions, in particular the Serre presentation of the quantum groups $U_\q(\mathfrak{gl}_N)$ and $U_\q(\mathfrak{sl}_N)$, the formulas for $R$-matrices, and the Peter-Weyl theorem.  Recall that a $U_\q(\mathfrak{g})$-module is called integrable if the Cartan generators $K_i$ act diagonalizably with eigenvalues in $\q^\ZZ$, and each vector lies in a finite-dimensional submodule.

  For each
$\lambda\in\gLatp$ or $\sLatp$, we denote by $V_\lambda$ the unique simple module of highest weight $\lambda$.
Note that we have an isomorphism $(V_\lambda)^* \cong V_{\lambda^*}$.

\subsection{The vector representation} The representation $V_{\ep{1}}\cong\CC^N$ for either $U_\q(\mathfrak{gl}_N)$ and $U_\q(\mathfrak{sl}_N)$ will be simply denoted $V$.  We fix $e_{1},\ldots, e_N$ to be the standard basis for $V$, and we denote by $\rho_V$ the associated homomorphism to $\operatorname{End}(V)$.  Recall that the $\GL_N$ $R$-matrix for the vector representation can be expressed explicitly:
\begin{equation}\label{eqn:R}
R:=(\rho_V\ot\rho_V)(\cR)
= \left(\q\sum_{i}E_i^i\otimes E_i^i
+\sum_{i\neq j}E_i^i\otimes E_j^j+
(\q-\q^{-1})\sum_{i>j}E_i^j\otimes E_j^i\right).
\end{equation}

We define $R^{ik}_{jl},(R^{-1})^{ik}_{jl}\in\CC$, for $i,j,k,l=1,\ldots,N$ by:
\begin{eqnarray*}
R(e_{i}\otimes e_{j})
=\sum_{k,l}R_{ij}^{kl}(e_{k}\otimes e_{l}),\quad
R^{-1}(e_{i}\otimes e_{j})
=\sum_{k,l}(R^{-1})_{ij}^{kl}(e_{k}\otimes e_{l}).
\end{eqnarray*}
We can write the coefficients explicitly as follows:
{\small
\begin{equation}\label{eqn:R4}
R_{ij}^{kl}=\left\{\begin{array}{ccc}\q, &  & i=j=k=l; \\1, &  & i=k\neq j=l; 
\\\q-\q^{-1}, &  & i=l<j=k; \\0, &  &\text{otherwise};\end{array}\right.\quad
(R^{-1})_{ij}^{kl}=\left\{\begin{array}{ccc}\q^{-1}, &  & i=j=k=l; 
\\1, &  & i=k\neq j=l; \\\q^{-1}-\q, &  & i=l<j=k; \\0, &  &\text{otherwise}.
\end{array}\right.
\end{equation}
}

Let $\tau:V\ot V$ denote the tensor flip, $\tau(v\ot w)=w\ot v$.  The braiding, $\sigma_{V,V}=\tau\circ R$, for $U_\q(\glN)$ satisfies a Hecke relation,
$$(\sigma_{V,V}-\q)(\sigma_{V,V}+\q^{-1})=0.$$

The $SL_N$ $R$-matrix on the vector representation is equal to the $GL_N$ $R$-matrix, multiplied by a factor of $\q^{-\frac{1}{N}}$.  To avoid confusion, we will reserve the notation $R$ for the $\GL_N$ $R$-matrix, and write $\q^{-\frac{1}{N}}R$ to reference the $SL_N$ $R$-matrix.  Hence the resulting braiding $\sigma_{V,V}=\tau\circ(\q^{-\frac{1}{N}} R)$ for $U_\q(\slN)$ satisfies the shifted Hecke relation,
$$(\sigma_{V,V}-\q^{-\frac{1}{N}}\q)(\sigma_{V,V}+\q^{-\frac{1}{N}}\q^{-1})=0.$$

\subsection{The determinant representation}  We will denote by $\detq(V)=\bigwedge\nolimits_q^N(V)$ the one-dimensional determinant representation of $U_\q(\glN)$, and we will abbreviate
$$\detq^k(V) = \left\{\begin{array}{ll}(\detq(V))^{\otimes k},& k\geq 0,\\ (\detq(V)^*)^{\ot k},& k<0.\end{array}\right.$$

\subsection{The ribbon element}
Recall the ribbon element\footnote{Note that axioms for the ribbon element and its inverse are swapped between \cite{Jordan2008} and \cite{OR}; we will follow the conventions of \cite{OR}.} $\nu\in U_\q(\mf g)$.  The ribbon element is central and satisfies the identity
\begin{equation}\Delta(\nu)=R_{21}R_{12}(\nu\otimes \nu).
\label{ribbon-id}\end{equation}
It acts on any irreducible representation $V_\lambda$, and hence on any isotypic component $X[\lambda]$, by the scalar
\begin{equation}
\label{ribbon-elt}
\nu|_{X[\lambda]} =\q^{\langle\lambda+2\rho,\lambda\rangle}.
\end{equation}

\subsection{The quantum coordinate algebra}

\begin{definition} The reflection equation algebra of type $\GL_N$, denoted $\cO_\q(Mat_N)$, is the algebra generated by symbols $a^i_j$, for $i,j = 1,\ldots N$, subject to the relations,
 $$R_{21}A_1R_{12}A_2 =  A_2R_{21}A_1R_{12},$$
where $A := \sum_{i,j} a^i_j E^j_i$ is a matrix with entries the generators $a^i_j$, and for a matrix $X$, we write $X_1=X\ot \id_V$, $X_2=\id_V\ot X$, so that the matrix equation above is equivalent to the list of relations, for $i,j,n,r\in\{1,\ldots,N\}$:
\begin{equation}\sum_{k,l,m,p}R^{ij}_{kl}a^l_mR^{mk}_{np}a^p_r = \sum_{s,t,u,v} a^i_sR^{sj}_{tu}a^u_vR^{vs}_{nr}\label{OqRelns}\end{equation}
\end{definition}

\begin{remark} We note that, since $R$ appears quadratically on both sides of the defining relation, the relations are unchanged by replacing $R \leadsto \q^{-\frac{1}{N}}R$.
\end{remark}

\begin{proposition}[\cite{JW}]
The element,
$$\detq(A) := \sum_{\sigma\in \SN} (-\q)^{\ell(\sigma)}\cdot \q^{e(\sigma)} a^1_{\sigma(1)}\cdots a^N_{\sigma(N)},$$
is central in $\cO_\q(Mat_N)$.
\end{proposition}
Here $\ell(\sigma)$ denotes the length, i.e. the number of pairs $i<j$ such that $\sigma(i)>\sigma(j)$, and $e(\sigma)$ denotes the excedence, i.e. the number of elements $i$ such that $\sigma(i)>i$.

\begin{definition}
The quantum coordinate algebras $\cO_\q(\GL_N)$ and $\cO_\q(SL_N)$ are the algebras obtained from $\cO_\q(Mat_N)$, by inverting, respectively specializing to one, the central element $\detq$.  That is,
$$\cO_\q(\GL_N) = \cO_\q(Mat_N)[\detq(A)^{-1}],\qquad \cO_\q(SL_N) = \cO_\q(Mat_N)/\langle \detq(A)-1 \rangle.$$
\end{definition}

\begin{theorem}[Peter-Weyl decomposition]\label{thm-Peter-Weyl}
As a module for $U_\q(\glN)$ and $U_\q(\slN)$, respectively, we have isomorphisms:
$$\cO_\q(\GL_N) \cong \bigoplus_{\lambda\in \gLatp} V_\lambda^* \ot V_\lambda, \qquad \cO_\q(SL_N) \cong \bigoplus_{\lambda\in \sLatp} V_\lambda^* \ot V_\lambda.$$
\end{theorem}

\subsection{Quantum differential operators}
The algebra of quantum differential operators on $G$, which we denote by $\cD_\q(G)$, was studied in many different settings.  The presentation below as a twisted tensor product is adapted from the paper \cite{VV} (see also \cite{BrochierJordan2014}), and hence matches the conventions of \cite{Jordan2008} (see footnote 3 there, however).
\begin{definition} For $G=\GL_N$, or $SL_N$, the algebra $\cD_\q(G)$ is the twisted tensor product,
\begin{gather}
\label{eq-Dq=Oq-Oq}
\cD_\q(G) = \cO_\q(G)\widetilde{\ot} \cO_\q(G),
\end{gather}
with cross relations,
\begin{align*}D_2R_{21}A_1 = R_{21}A_1R_{12}D_2R_{21}, &\qquad \textrm{if $G=\GL_N$}\\
D_2R_{21}A_1 = R_{21}A_1R_{12}D_2R_{21}\q^{-2/N}, &\qquad \textrm{if $G=SL_N$}
\end{align*}
where $A= \sum_{i,j}a^i_jE^j_i$ and $D=\sum_{i,j}\partial^i_j E^j_i$ denote matrices of generators of each tensor factor, so that the matrix equation above is equivalent to the list of cross relations, for $i,k,l,n\in\{1,\ldots,N\}$:
$$\sum_{j,m} \partial^i_jR^{jk}_{lm}a^m_n = \sum_{p,r,s,t,u,v}R^{ik}_{pr}a^r_sR^{sp}_{tu}\partial^u_vR^{vt}_{ln}.$$
\end{definition}

We denote by $\fun$ and $\partial_{\lhd}$ the inclusions into the first and second tensor factor of \eqref{eq-Dq=Oq-Oq}, so that $\fun\ot\partial_\lhd:\cO_\q(G)\ot\cO_\q(G)\to\cD_\q(G)$ is the tautological isomorphism of $U_\q(\g)$-modules (however it is not an algebra homomorphism).  This is a $\q$-deformation of the tensor decomposition $\cD(G) \cong \cO(G)\ot U(\g)$ into functions on $G$, and the vector fields of left-translation, hence the notation.

\section{Double affine Hecke algebras of type $A$}
\label{sec-DAHA}
Let $\K$ denote a field of characteristic zero, and let $q,t\in \K^\times$, and assume neither $q$ nor $t$ is a root of unity.  Typical instances are $\K=\CC$, $\CC(t)$, or $\CC(q,t)$.

\subsection{The extended affine symmetric group}

\begin{definition} The extended affine symmetric group is\footnote{We drop the first relation when $n=2$.}
$$\AffSym = \left\langle \pi, s_i, \,\,  i\in\ZZ/ n\ZZ\quad\Big|\quad
\begin{array}{ll}
s_is_{i+1}s_i = s_{i+1}s_is_{i+1} &\textrm{ for $i\in\ZZ/n\ZZ$},\\
s_is_j=s_js_i &\textrm{for $j \not\equiv i \pm 1 \bmod n$},\\
\pi s_i = s_{i+1}\pi &\textrm{for $i\in\ZZ/n\ZZ$,}
\\
s_i^2  = 1 &\textrm{ for $i\in\ZZ/n\ZZ$}
\end{array}\right\rangle.$$ \end{definition}

We can associate to $\SL_n$ the quotient $\AffSymSL = \AffSym/\langle \pi^n \rangle$.
Then we can think of the image, $\overline \pi$, of $\pi$ as the Dynkin diagram
automorphism or the generator of the cyclic group $\sLat/Q$.
Note both $\AffSym$ and $\AffSymSL$ have as a subgroup the affine symmetric group
$\langle s_i \mid i \in \ZZ/n\ZZ \rangle$.

We recall that $\AffSym$ acts on $\ZZ$ by $n$-periodic permutations, i.e. bijections $\sigma:\ZZ\to\ZZ$ such that $\sigma(i+n) = \sigma(i)+n$.  It also acts on the set $(\K^\times)^n$ via:
\begin{align}\label{AffSymmOnGLWts}
s_i\cdot(a_1,\ldots,a_i,a_{i+1},\ldots a_n) &= (a_1, \ldots, a_{i+1}, a_i, \ldots, a_n)\nonumber
\\
s_0\cdot(a_1,a_2,\ldots, a_{n-1},a_n) &=  (q a_n,  a_2, \ldots, a_{n-1}, q^{-1} a_1)
\\
\pi \cdot (a_1, \ldots, a_n) &= (q a_n, a_1, a_2, \ldots, a_{n-1})\nonumber
\end{align}

This action is relevant to the $GL_n$ double affine Hecke algebra (see Section \ref{sec-DahaGL}).  We modify the action of $\AffSym$ on $(\K^\times)^n$ for $\SL_n$ as follows, so that it will factor through the quotient $\AffSymSL$.
Let $\Sq \in\K^\times$ be another constant which we assume is not a root of unity.
\begin{align}
\label{AffSymmOnSLWts}
s_i\cdot(z_1,\ldots,z_i,z_{i+1},\ldots z_n) &= (z_1,\ldots,z_{i+1},z_{i},\ldots z_n),\nonumber\\
s_0\cdot(z_1,z_2,\ldots, z_{n-1},z_n) &=
	(\Sq^{-2n}z_n,z_2,\ldots, z_{n-1},\Sq^{2n}z_1),\\
\pi \cdot(z_1,\ldots, z_n) &= (\Sq^{-2n+2}z_n, \Sq^2z_1,\Sq^2z_2,\ldots, \Sq^2z_{n-1}).\nonumber
\end{align}
It is easy to check that in this case the action of the generators satisfy the
defining relations of $\AffSym$, as well as the additional relation
$\pi^n=\id$. 

See  Section \ref{sec:GLtoSLDAHA}
for a discussion of the relationship between $\Sq^{-2n}$ and $q$
that unifies the action of $s_0$ in \eqref{AffSymmOnGLWts} with that
in \eqref{AffSymmOnSLWts}.

\subsection{The elliptic braid group}
\label{sec-elliptic-braid-group}

\begin{definition} The \emph{elliptic braid group} $B_n^{Ell}$ is the fundamental group of the configuration space of $n$ points on the torus $T^2$.  The \emph{marked elliptic braid group} $B_{n,1}^{Ell}$ is the fundamental group of the configuration space of $n$ points on the punctured torus $T^2\backslash D^2$.
\end{definition}

\begin{proposition}[\cite{Birman,Scott}]
\label{prop-braid}
The elliptic braid group $B_n^{Ell}$ is presented by:
\begin{itemize} 
\item Pairwise commuting generators $\brX_1,\ldots, \brX_n$,
\item Pairwise commuting generators $\brY_1,\ldots, \brY_n$,
\item The braid group of the plane,
$$B_n = \left\langle \quad \brT_1, \ldots, \brT_{n-1} \quad \left| \begin{array}{cl}\brT_i\brT_{i+1}\brT_i = \brT_{i+1}\brT_i\brT_{i+1}, &i=1,\ldots n-2,\\ \brT_i\brT_j=\brT_j\brT_i, &|i-j|\geq 2\end{array}\right.\right\rangle,$$
\end{itemize}
with the cross relations:
\begin{gather*}
\brT_i\brX_i\brT_i = \brX_{i+1},\,\, \brT_i\brY_i\brT_i=\brY_{i+1}, \quad i=1,\ldots,n-1,\\
\brX_1\brY_2=\brY_2\brX_1\brT_1^2,\qquad (\prod_i\brX_i)\brY_j = \brY_j(\prod_i\brX_i), \quad j=1,\ldots, n,\\
\brX_i\brT_j=\brT_j\brX_i,\,\, \brY_i\brT_j=\brT_j\brY_i, \quad \textrm{ for $|i-j|>1$}.\end{gather*}
\end{proposition}

It is well-known that $B_{n,1}^{Ell}$ embeds into $B_{n+1}^{Ell}$, as the subgroup generated by $\brX_1,\ldots,\brX_n,$ $\brY_1,\ldots,\brY_n,$ $\brT_1,\ldots, \brT_{n-1}$, and by $\brT_n^2$.

\subsection{Double affine Hecke algebra for $\GL_n$}
\label{sec-DahaGL}

\begin{definition}
The $\GL_n$ double affine Hecke algebra $\HH_{q,t}(\GL_n)$ is
the $\K$-algebra presented by generators:
$$T_0,T_1,\ldots T_{n-1}, \pi^{\pm 1}, Y_1^{\pm 1},\ldots, Y_n^{\pm 1},$$
subject to relations\footnote{As with $\AffSym$, we drop the relations on the 
second line when $n=2$.}:
\begin{align} 
&(T_i-t)(T_i+t^{-1})=0 \quad (i=0,\ldots, n-1),& && \label{HeckeReln}\\
&T_iT_jT_i = T_jT_iT_j\quad (j\equiv i\pm 1 \bmod n),& &T_iT_j = T_jT_i \quad (\textrm{otherwise}),&\label{BraidReln}\\
&\pi T_i\pi^{-1} = T_{i+1} \quad (i=0,\ldots, n-2),& &\pi T_{n-1}\pi^{-1}=T_0,&\nonumber\\
&T_iY_iT_i=Y_{i+1} \quad (i=1,\ldots, n-1),& &T_0Y_nT_0 = q^{-1}Y_1&\nonumber\\
&T_i Y_j = Y_j T_i \quad  (j \not\equiv i, i+1 \bmod n),& &&\nonumber\\
&\pi Y_i\pi^{-1} = Y_{i+1} \quad (i=1,\ldots, n-1),& &\pi Y_{n}\pi^{-1}=
 q^{-1}Y_1&\nonumber
\end{align}

\end{definition}

Any $\sigma\in\AffSym$ has a canonical lift $T_\sigma\in\HH_{q,t}(\GL_n)$, defined as follows: if $\sigma$ is written as a
reduced word $\sigma=\pi^r s_{i_1}\cdots s_{i_k}$ of the generators, then we set $T_\sigma = \pi^r T_{i_1}\cdots T_{i_k}$.  This expression for $T_\sigma$ is well-defined because the $T_i$ satisfy the same braid relations as $s_i \in \AffSym$.  We abuse notation and abbreviate $\pi = T_\pi$.  For $\beta = (b_1, b_2, \ldots, b_n) \in \ZZ^n$ we denote by
$Y^\beta = Y_1^{b_1} \cdots Y_n^{b_n}$ the corresponding monomial.
We note that $\HG$ has basis $\{ T_{\sigma} Y^\beta \mid \sigma \in \AffSym, \beta \in \ZZ^n \}$.

Given the combinatorial viewpoint of this paper, the presentation above involving $\pi$ is the most convenient for us. However, it is sometimes desirable to define
\begin{gather*}
X_1 = \pi T_{n-1}^{-1} \cdots T_{2}^{-1} T_{1}^{-1},\qquad
X_{i+1} = T_i X_i T_i \quad (i=1, \ldots, n-1).
\end{gather*}
Then it is not hard to show $X_i X_j =X_j X_i$ and that these
elements generate a Laurent polynomial subalgebra
$\K[X_1^{\pm 1},\ldots, X_n^{\pm 1}] \subseteq \HH_{q,t}(\GL_n)$.
It is also easy to show that $X_1 X_2 \cdots X_n = \pi^n$, and that this element $q$-commutes with each $Y_i$.

We thus have two commutative sub-algebras,
$$\X = \K[X_1^{\pm 1},\ldots, X_n^{\pm 1}] \qquad \textrm{ and }\qquad \Y = \K[Y_1^{\pm 1},\ldots, Y_n^{\pm 1}],$$
of $\HH_{q,t}(\GL_n)$, each isomorphic to a Laurent polynomial ring.

Further, the DAHA $\HG$ has two distinguished subalgebras
$$H(\Y) = \langle T_1, \ldots, T_{n-1}, Y_1^{\pm 1},\ldots, Y_n^{\pm 1} \rangle,\qquad
H(\X) = \langle T_0, T_1, \ldots, T_{n-1}, \pi^{\pm 1} \rangle,$$ each of which is isomorphic to the extended affine Hecke algebra of type $A$.

Finally, we will denote by $H_n$ the finite Hecke algebra, generated by $T_i$ and subject to the relations \eqref{HeckeReln} and \eqref{BraidReln}.  The finite Hecke algebra sits as a common subalgebra of $H(\X)$ and $H(\Y)$, but it is also naturally realized as a quotient
of $H(\Y)$ via the homomorphism
$H(\Y) \to H_n$
determined by
\begin{gather}
\label{eq-ev}
T_i \mapsto T_i,\qquad
Y_1 \mapsto 1.
\end{gather}
In this way we can inflate any $H_n$-module to be a $H(\Y)$-module.

We note that the center $Z(H(\Y)) = \K[Y_1^{\pm 1}, \ldots, Y_n^{\pm 1}]^{\Sn}$
consists of the symmetric Laurent polynomials. 
In particular the product $Y_1 Y_2 \cdots Y_n$ commutes with all $T_i$.
(It even commutes with $T_0$ which is not in $H(\Y)$, but
does not commute with $\pi$.)

\begin{proposition}
\label{prop-braid-to-DAHA-GL}
There exists a unique isomorphism,
$$\phi:\K[B_{n,1}^{Ell}]\Big/\left\langle \brT^2_n=q,\qquad (\brT_i-t)(\brT_i+t^{-1})=0\quad (i=1,\ldots, n-1) \right\rangle \to \HH_{q,t}(\GL_n),$$
such that
$$\phi(\brT_i)=T_i, \quad (i=1,\ldots n-1),\qquad \phi(\brX_1) = \pi T_{n-1}^{-1} \cdots T_{2}^{-1} T_{1}^{-1}, \qquad \phi(\brY_1) = Y_1.$$
\end{proposition}

\begin{proof}
We shall mostly leave this to the reader, but let us explain the relation $\brT^2_n=q$ here.  In $B_{n+1}^{Ell}$, we have:
$\brX_{n+1}\brY_{n}=\brT_n^2\brY_{n}\brX_{n+1}$.
Recall too that the $\brX_i$ commute with each other and further
$\brT_n$ commutes with $\brX_n \brX_{n+1}$.
Hence
\begin{eqnarray*}
\brX_{n+1} \brY_n (\brX_1 \cdots \brX_n) &=& \brT_n^2\brY_{n}\brX_{n+1} (\brX_1 \cdots \brX_n) \\
&=& \brT_n^2\brX_1 \cdots \brX_n\brX_{n+1} \brY_{n}\\
&=& \brX_1 \cdots \brX_n\brX_{n+1} \brT_n^2 \brY_{n}\\
&&\text{multiplying both sides on the left by $\brX_{n+1}^{-1}$ gives}\\
 \brY_n (\brX_1 \cdots \brX_n) &=&  (\brX_1 \cdots \brX_n) \brT_n^2\brY_{n}.
\end{eqnarray*}
It is easy to check $\phi(\brX_1 \cdots \brX_n) = \pi^n$.  Hence setting $\phi(\brT_n^2)=q$ is consistent with our $\pi^n q Y_n \pi^{-n} = Y_n$ relation. 
We leave it to the reader to check the other relations.
\end{proof}

\subsection{Double affine Hecke algebra for $\SL_n$}
Let us fix further constants $\Yprod, \Sq \in\K^\times$, which we assume are not roots of unity.
However, for convenience, we will assume $\K^\times$ contains primitive $n$th roots of unity.

\begin{definition}
The $\SL_n$ double affine Hecke algebra $\HH_{\Sq,t}(\SL_n)$ is presented by generators:
$$T_0,T_1,\ldots T_{n-1}, \pi^{\pm 1}, Z_1^{\pm 1},\ldots, Z_n^{\pm 1},$$
subject to relations:
\begin{align*} 
&(T_i-t)(T_i+t^{-1})=0 \quad (i=0,\ldots, n-1),& &&\\
&T_iT_jT_i = T_jT_iT_j\quad (j\equiv i\pm 1 \bmod n),& &T_iT_j = T_jT_i \quad (\textrm{otherwise}),&\\
&\pi T_i\pi^{-1} = T_{i+1} \quad (i=0,\ldots, n-2),& &\pi T_{n-1}\pi^{-1}=T_0,&\\
&T_iZ_iT_i=Z_{i+1} \quad (i=1,\ldots, n-1),&
&T_i Z_j = Z_j T_i \quad  (j \not\equiv i, i+1 \bmod n),&
\\
&T_0Z_nT_0 = \Sq^{2n}Z_1,&
&Z_1 Z_2 \cdots Z_n =  \Sq^{n(n-1)}\Yprod^n, &
& \pi^n=1,&\\
&\pi Z_i\pi^{-1} = \Sq^{-2} Z_{i+1} \quad (i=1,\ldots, n-1),& &\pi Z_{n}\pi^{-1}=
 \Sq^{2n-2}Z_1.&
\end{align*}
\end{definition}

\medskip
Similar to the $GL$ case, we have two commutative subalgebras,
$$\X = \langle X_1^{\pm 1},\ldots, X_n^{\pm 1}\rangle \qquad \textrm{ and }\qquad \Z = \langle Z_1^{\pm 1},\ldots, Z_n^{\pm 1}\rangle$$
of $\HH_{\Sq,t}(\SL_n)$, each isomorphic to the quotient of a Laurent polynomial ring by a single relation on the product of generators, as in Proposition \ref{SLnBraidtoDAHA} below.
We let $H(\Z)$ denote the subalgebra generated by $\Z$ and $H_n$. 
We can realize $H(\Z)$ as a quotient of the extended affine Hecke
algebra by the
relation $Y_1\cdots Y_n= \Sq^{n(n-1)}\Yprod^n$,
\begin{align*}H(\Y)&\to H(\Z),\\
T_i&\mapsto T_i &\qquad (i=1, \ldots, n-1) \\
Y_i&\mapsto  Z_i&\qquad (i=1, \ldots, n). \\
\end{align*}

\begin{definition}
\label{def-twist}
Let $a \in \K^\times$ satisfy $a^n = 1$.
Given an $\HS$-module $M$, denote by $M^a$ the  \emph{twist}
 of $M$ by the automorphism of $\HS$ that sends 
$T_i \mapsto T_i$, $Z_i \mapsto Z_i$, $\pi \mapsto a \pi$.
\end{definition}

It may happen for some modules $M$ that $M \cong M^a$.
See Section \ref{sec-SLpsyt}.

\begin{remark}
For any $a,b \in \K^\times$, simultaneously rescaling all $X_i\mapsto a X_i$ and all $Y_i\mapsto b Y_i$ defines an automorphism of $\HH_{q,t}(GL_n)$, and $\K[B_n^{Ell}]$, compatible with $\phi$.

For $\HH_{\Sq,t}(\SL_n)$, however, rescaling $\pi\leadsto a\pi$ and $Z_i\leadsto bZ_i$'s 
changes the relations by $(\pi^n=1)\leadsto (\pi^n=a^n)$, and $\Yprod\leadsto b\Yprod$.  Hence, in order to define an isomorphism $a$ must be an $n$th root of unity as in Definition \ref{def-twist} above.  On the other hand, this shows that the parameter $\Yprod$ is inessential, and can be chosen at will by re-scaling the $Z_i$'s.  In Section \ref{sec-rect}, we will take $n$ to be a multiple of a fixed integer $N$, and will fix $\Yprod=\Sq^{1-N^2}=\q^{1/N-N}$, the inverse value of the ribbon element on the defining representation of $\SLN$.  This is a matter of combinatorial convenience.  One could also introduce such an inessential parameter for the value of $\pi^n$, but $\pi^n=1$ is the most natural choice combinatorially, so we fix this convention now.  This forces the scalar factor in front of $\phi(\brX_1)$ below.
\end{remark}

\begin{proposition}\label{SLnBraidtoDAHA}
There exists a unique isomorphism,
$$\phi: \K[B_n^{Ell}]\Big/\left\langle \begin{array}{c} (\brT_i-\Sq^{-1}t)(\brT_i+\Sq^{-1}t^{-1})=0 \quad(i=1,\ldots, n-1),\\\brX_1\cdots \brX_n=\Yprod^n,\quad \brY_1\cdots \brY_n=\Yprod^n\end{array} \right\rangle \to \HH_{\Sq,t}(SL_n),$$
such that
$$\phi(\brT_i)=\Sq^{-1}T_i, \quad (i=1,\ldots n-1),\qquad \phi(\brX_1) = \Yprod\cdot \Sq^{n-1}\cdot \pi T_{n-1}^{-1} \cdots T_{2}^{-1} T_{1}^{-1}, \qquad \phi(\brY_1) = Z_1.$$
\end{proposition}

We note that this implies $\phi(\brY_i) = \Sq^{2(1-i)} Z_i$.

\begin{remark}
In \cite{Jordan2008}, a variant of Proposition \ref{SLnBraidtoDAHA} is taken as the definition of $\HH_{\Sq,t}(SL_n)$, because of the central role played in the paper by the elliptic braid group.  
\end{remark}

\subsection{Relationship between DAHAs for  $\GL_n$ and $\SL_n$}\label{sec:GLtoSLDAHA}
The relationship between the $\GL_n$ and $\SL_n$ DAHA is not entirely straightforward:  the $\SL_n$ DAHA may be realized as a sub-quotient of a degree $n$-extension of $\GL_n$ DAHA.  This can be found in \cite{CherednikBook}, but we review it here in the notation of this paper.

First we note that in $\HS$, we have the relation $\pi^n=1$, while in $\HG$, we cannot simply set $\pi^n=1$, because of the relation $\pi^n Y_i \pi^{-n} = q^{-1} Y_i.$  However, we note that $\pi^n$ commutes with the $T_i$'s and with the Laurent polynomials in the $Y_i$'s of total degree zero.  We enlarge the degree zero part of $\Y$,
by adjoining an element,
$$\Yt\quad ``="\quad (Y_1 \cdots Y_n)^{-1/n}.$$
More precisely, this means we define a new algebra $\widetilde{\HH_{q,t}}(\GL_N)$, by adding to $\K$, if necessary, a $2n$th root of $q$, and adjoining to $\HH_{q,t}(\GL_N)$ an element $\Yt$ in degree $-1$, such that $\Yt$ commutes with all $T_i, \quad (i=0\ldots n-1),$ and all $Y_i, \quad(i=1,\ldots, n)$, and subject to the further relations:
$$\Yt^n=(Y_1 \cdots Y_n)^{-1},\qquad \pi \Yt \pi^{-1} = q^{1/n} \Yt.$$
We let $\cYt$ denote the resulting subalgebra of degree $0$ polynomials, i.e., the subalgebra generated by
$(Y_i \Yt)^{\pm 1}$.
The algebra generated by $\pi, T_1, \cYt$ now has $\pi^n-1$ in its center,
so we may quotient by the ideal it generates (and below mark
elements of the quotient with a bar).
We can identify $\HH_{q^{-1/2n},t}(\SL_n)$ with this quotient via
\begin{align*}
T_i \mapsto \overline {T_i},\qquad
Z_i \mapsto q^{\frac{1-n}{2n}}\Yprod \overline {Y_i \Yt},\qquad
\pi \mapsto \overline \pi.
\end{align*}

Hence we can identify the parameter
$\Sq$ of $\HS$ with $q^{-1/2n}$ of $\HG$. 
To avoid fractional exponents, we chose to make $\Sq$ a separate parameter.

\section{Category $\cO$ and $\Y$-semisimple representations}
\label{sec-Yssl}
In his hallmark paper \cite{Cherednik03},
Cherednik gave a complete classification of irreducible $\Y$-semisimple
representations, i.e. those $\HH_{q,t}$-modules 
for which the $\Y$-action can be diagonalized.
  His classification builds on  the parallel story for the affine Hecke
algebra \cite{emaila,emailb,emailc}, \cite{Ram, Ram2}.
 Subsequently, the paper \cite{SV} built on Cherednik's classification via
periodic skew diagrams combinatorially, connecting standard tableaux on the
diagrams to $\Y$-weights.  We recall these constructions below.  

We will give basic definitions in parallel for both $GL$ and $SL$. 
Theorems \ref{thm-mult-one} and \ref{thm-Yssl-gens} were proved in
\cite{emaila,emailb,emailc}, \cite{Ram} for affine Hecke
algebras and in \cite{Cherednik-Selecta, Cherednik03}, \cite{SV} for
$\HH_{q,t}(GL_N)$.  We will also state without proof the
analogous results for $SL_N$, which follow by straightforward
modifications of the proofs, keeping track of the additional relation
$\pi^n=1$.  We will  write $\HH_{q,t}$ in this section when we do not
need to distinguish between the $GL$ or $SL$ case.

\begin{definition}
Category $\cO$ for $\HH_{q,t}(\GL_n)$ (resp. $\HH_{\Sq,t}(\SL_n)$) is the full subcategory of finitely generated $\HH_{q,t}$-modules $M$, such that for each vector $m\in M$, its orbit $\Y\cdot m$ (resp. $\Z\cdot m$) is finite dimensional.
\end{definition}

A tuple $\uz = (z_1,\ldots, z_n)\in (\K^\times)^n$ will be called a weight for $\Y$.  A tuple $\uz$ satisfying further that $\prod_iz_i=\Sq^{n(n-1)}\Yprod^n$ will be called a $\Z$-weight.  We will make use of the actions \eqref{AffSymmOnGLWts} and \eqref{AffSymmOnSLWts} of the extended affine symmetric group on each type of weights.

A $\Y$-weight $\uz$ defines a homomorphism $\Y\to\K$, sending $Y_i\mapsto z_i$, and hence a one-dimensional representation of $\Y$.  Similarly a $\Z$-weight $\uz$ defines a homomorphism $\Z\to\K$,
sending $Z_i\mapsto z_i$, and hence a one-dimensional representation of $\Z$.  

Let $\M$ be an $\HH_{q,t}$-module.  We define its $\uz$-weight space to be,
$$\M[\uz] = \{ v\in\M \,\,|\,\, hv=\uz(h)v,\,\, \forall h\in \Y \textrm{ (resp. } \Z)\}.$$
A non-zero $v\in\M[\uz]$ is called a weight vector, or $\uz$-weight vector.  Its generalized weight space is
$$\M^{gen}[\uz] = \{v\in \M \,\,|\,\, (h-\uz(h))^mv=0, \textrm{ for all }
h\in \Y \textrm{ (resp. } \Z), m\gg 0 \}$$

Assume for any $M \in \cO$ that $M$ splits over $\K$.  Then we have an isomorphism,
$$\M \cong \bigoplus_{\uz}\M^{gen}[\uz],$$
as vector spaces, and in fact as $\Y$- or $\Z$-modules.

 We will call
$$\supp(\M) = \{\uz \mid \M^{gen}[\uz]\neq 0\} = \{\uz \mid \M[\uz]\neq 0\},$$
the support of $\M$.

\begin{definition} We say that $\M$ is \emph{$\Y$-semisimple} (resp. $\Z$-semisimple) if we have an isomorphism
$$\M\cong \bigoplus_{\uz}\M[\uz],$$ as $\Y$- (resp. $\Z$-) modules, i.e. if $\Res^{\HH_{q,t}}_{\Y}(\M)$ is semisimple as a $\Y$- (resp. $\Z$-) module. 
We will write ``$\YZ$-semisimple" to mean either $\Y$- or $\Z$-semisimple, depending on $G$.  \end{definition}
Such $\M$ are called calibrated in \cite{Ram}.  Note that if $\M$ is $\YZ$-semisimple, it has a \emph{weight basis}: a basis consisting of weight vectors.

The notion of $\YZ$-semisimplicity makes sense whether $\M$ is a representation of the affine Hecke algebra or of the double affine Hecke algebra.  The structure of $\YZ$-semisimple modules in both cases is extremely rigid.  Given a $\Y$-semisimple module, one can read its composition factors directly from its support.  Further if $\M$ is both simple and $\Y$-semisimple, one need only determine a single $\uz\in\supp(\M)$ in order to determine all of $\supp(\M)$, and hence the isomorphism type of $\M$.
For simple and $\Z$-semisimple $\HS$-modules, similar statements hold, once we allow twisting by an automorphism that rescales $\pi$ by a root of
unity.

\begin{theorem}[\cite{emaila,emailb,emailc,Cherednik-Selecta},
\cite{Ram},\cite{SV}] \label{thm-mult-one}
Let $\M$ be a simple and $\YZ$-semisimple $\HH_{q,t}$-module. Then for all $\uz\in\supp(\M)$, we have $\dim\M[\uz]=1$.
\end{theorem}
Hence, a simple and $\YZ$-semisimple $\HH_{q,t}$-module has a weight basis that is unique up to rescaling each vector individually, i.e. the underlying vector space has a canonical decomposition as a direct sum of lines.

For $\HS$ one must modify the proof of Theorem \ref{thm-mult-one}
for $\HG$ to include the case that the stabilizer of $\uz$
via the action  \eqref{AffSymmOnSLWts}
contains an element outside of $\langle s_0, s_1, \ldots, s_{n-1} \rangle$.
In that case it contains a subgroup conjugate to $\langle \pi^r \rangle$ for
some $r \mid n$.  One may then appeal to minimal idempotents for
the cyclic group $\langle \pi^r \rangle \big/ \langle \pi^n \rangle$
to get the required multiplicity one result.

The supports of simple $\YZ$-semisimple modules have a very nice combinatorial structure.  It is easy to show that if $\M$ is simple, then all its support is contained in a single $\AffSym$-orbit.  If additionally $\M$ is $\YZ$-semisimple, we can say \emph{exactly} what subset of the $\AffSym$-orbit we get, i.e. we can completely determine $\supp(\M)$. More precisely, given $\uz\in\supp(\M)$, one can determine the set $S\subset\AffSym$ such that
$$\supp(M) = \{w\cdot\uz \,\, | \,\, w\in S\}.$$
The following theorem uniquely characterizes the set $S$, which depends
on choice of $\uz$:

\begin{theorem}[\cite{emaila,emailb,emailc,Cherednik03}, \cite{Ram}]
\label{thm-Yssl-gens} Let $\M$ be a simple and $\YZ$-semisimple $\HH_{q,t}$-module.  Let $\uz\in\supp(\M)$. We have:
\begin{enumerate}
\item
\label{item-Yss-one}
 For $1\leq i < n,$ we have $\M[s_i\cdot\uz]=0$ if, and only if, $z_i/z_{i+1}\in \{t^2,t^{-2}\}$.  Further, we have:
$$T_i\M[\uz] \subset \M[\uz] \oplus \M[s_i\cdot\uz].$$
\item
\label{item-Yss-two}
We have $\M[s_0\cdot\uz]=0$ if, and only if
$ q z_n/z_1\in\{t^2,t^{-2}\}$ for $\HG$,
$ z_n/\Sq^{2n}z_1\in\{t^2,t^{-2}\}$ for $\HS$. 
Further, we have:
$$T_0 M[\uz] \subset \M[\uz]\oplus \M[s_0\cdot \uz].$$
\item We have $\M[\pi\cdot\uz]\neq 0$, and $\pi\M[\uz] = \M[\pi\cdot\uz]$.
\end{enumerate}
\end{theorem}

\begin{remark} Conditions \eqref{item-Yss-one}
and \eqref{item-Yss-two} 
of Theorem \ref{thm-Yssl-gens} may be written more uniformly as
$\frac { (s_i \cdot \uz)_{i+1}} {z_{i+1}} \in \{t^2, t^{-2}\}$,
but not much clarity is gained from this reformulation.
\end{remark}

Note that Theorem \ref{thm-Yssl-gens} allows us to precisely describe the
action of the $\HG$-generators on a weight basis, once we have chosen a
sensible normalization/scaling.  The proof of this theorem uses the theory
of ``intertwiners" \cite{Cherednik03}, for which the reader may
also consult \cite{SV}.  For $\HS$, this is true once we pin down the
action of $\pi$ (see Corollary \ref{cor-SLwtpi}).

As a consequence of this rigid structure, we also have the following:

\begin{corollary}
\label{cor-Yssl}
Let $\M,\N$ be simple $\HG$-modules, which are both $\Y$-semisimple.  Then either $\M\cong\N$, or $\supp(\M)\cap\supp(\N)=\emptyset$.
\end{corollary}

\begin{corollary}\label{cor-SLwtpi}
Let $\M,\N$ be simple $\HS$-modules, which are both $\Z$-semisimple.  Then
either $\M\cong\N^a$ for some $a \in \K^\times$ with $a^n=1$,
 or $\supp(\M)\cap\supp(\N)=\emptyset$. 
\end{corollary}

We will use these corollaries to identify the module $F^G_n(\Oq)$ in Section \ref{sec-main-results}. 

\section{The rectangular representations}
\label{sec-rect}
In this section we will detail a very special case of Cherednik's
construction, when the Young diagram indexing the irreducible module is an
$N\times k$ rectangle,
 $$\mu= (k^N) = (\underbrace{k,\ldots, k}_{N}).$$
and the periodicity is purely horizontal, so that the shape is not actually skew.  We begin, however, by recalling the simpler setting of the affine Hecke algebra.

\subsection{The rectangle and the affine Hecke algebra}
\label{sec-AHA-rect}
Associated to the partition $\mu=(k^N)$ is a finite dimensional irreducible representation of the finite Hecke algebra $H_n$, with $n=kN$.  A basis for this representation is indexed by the set of standard Young tableaux of shape $(k^N)$.  We denote by $\rect$ the $H(\Y)$-module obtained by inflating this module via the homomorphism \eqref{eq-ev}.  It is well-known for generic $t$ (i.e. away from small roots of unity) that $\rect$ is $\Y$ semisimple.
This is stated precisely in Proposition \ref{prop-AHARect} below.
Identifying the $N \times k$ rectangular diagram $(k^N)$ with $\DNk{N}{k}{0}$,
i.e. $\lambda = 0$, we assign diagonal labels as in Section \ref{sec-typeA}.
 Hence its principal
diagonal is labeled $0$, and a box in the $j$th row and $m$th column is on the
$m-j$ diagonal.  This agrees with the traditional notion of content.

\begin{definition}
\label{def-diagRect}
Given a standard tableau $R$ of shape $(k^N)$, we define a map $\diag_R : \{1, \ldots, n\} \to \ZZ$ such that $\diag_R(i)$ is the label of the diagonal on which $\squarebox{i}$ lies.\end{definition}

\begin{figure}[h]\label{content-fig}
\begin{tikzpicture}[scale=0.6]
\begin{scope}[shift={(5,1)}]
\draw [thick] (0,0) rectangle (4,3);
\draw (0.5,0.5) node[blue] {$-2$};
\draw (0.5,1.5) node[blue] {$-1$}; 
\draw (0.5,2.5) node[blue]  {$0$}; 
\draw (1.5,0.5) node[blue] {$-1$}; 
\draw (1.5,1.5) node[blue] {$0$}; 
\draw (1.5,2.5) node[blue] {$1$}; 
\draw (2.5,0.5) node[blue] {$0$}; 
\draw (2.5,1.5) node[blue] {$1$}; 
\draw (2.5,2.5) node[blue] {$2$}; 
\draw (3.5,0.5) node[blue] {$1$}; 
\draw (3.5,1.5) node[blue] {$2$}; 
\draw (3.5,2.5) node[blue] {$3$}; 
\draw [dashed,blue] (-0.5,3.5) -- (3.5, -0.5);
\end{scope}
\end{tikzpicture}
\caption{The diagonal labels of a rectangular tableau.  The dashed line traverses the principal diagonal.}
\end{figure}

\begin{definition}
\label{def-weightRect}
The \emph{weight} $\wtR(R) \in (\Kt)^n$ of $R\in \sytk$, is the tuple,
$$\wtR(R) = \left(t^{2 \diag_R(1) },t^{2 \diag_R(2) },\ldots ,t^{2 \diag_R(n) }\right) = t^{\left(2 \diag_R(1) , 2 \diag_R(2) ,\ldots ,2 \diag_R(n)\right) }.$$
\end{definition}

\begin{example}
Let $N=3, k=2$. There are $5$ standard tableau of shape $(2^3)$.  We list them below with their corresponding weights.

 \begin{tikzpicture}[scale=0.6]

\begin{scope}[shift={(0,2)}]
\draw [thick] (0,0) rectangle (2,3);
\draw[step=1] (0,0) grid (2,3);
\draw (0.5,0.5) node {$3$};
\draw (0.5,1.5) node {$2$}; 
\draw (0.5,2.5) node  {$1$}; 
\draw (1.5,0.5) node {$6$}; 
\draw (1.5,1.5) node {$5$}; 
\draw (1.5,2.5) node {$4$}; 
\draw [dashed,blue] (-0.25,3.25) -- (2.5, .5);
\draw (-0.5,3.5) node[blue] {\scriptsize $0$};
\end{scope}
\begin{scope}[shift={(0,-1)}]
\draw (1,2.5) node  {$t^{(0,-2,-4,2,0,-2)} $}; 
\end{scope}

\begin{scope}[shift={(5,2)}]
\draw [thick] (0,0) rectangle (2,3);
\draw[step=1] (0,0) grid (2,3);
\draw (0.5,0.5) node {$4$};
\draw (0.5,1.5) node {$2$}; 
\draw (0.5,2.5) node  {$1$}; 
\draw (1.5,0.5) node {$6$}; 
\draw (1.5,1.5) node {$5$}; 
\draw (1.5,2.5) node {$3$}; 
\end{scope}
\begin{scope}[shift={(5,-1)}]
\draw (1,2.5) node {$t^{(0,-2,2,-4,0,-2)} $}; 
\end{scope}

\begin{scope}[shift={(10,2)}]
\draw [thick] (0,0) rectangle (2,3);
\draw[step=1] (0,0) grid (2,3);
\draw (0.5,0.5) node {$4$};
\draw (0.5,1.5) node {$3$}; 
\draw (0.5,2.5) node  {$1$}; 
\draw (1.5,0.5) node {$6$}; 
\draw (1.5,1.5) node {$5$}; 
\draw (1.5,2.5) node {$2$}; 
\end{scope}
\begin{scope}[shift={(10,-1)}]
\draw (1,2.5) node {$t^{(0,2,-2,-4,0,-2)}$}; 
\end{scope}

\begin{scope}[shift={(15,2)}]
\draw [thick] (0,0) rectangle (2,3);
\draw[step=1] (0,0) grid (2,3);
\draw (0.5,0.5) node {$5$};
\draw (0.5,1.5) node {$2$}; 
\draw (0.5,2.5) node  {$1$}; 
\draw (1.5,0.5) node {$6$}; 
\draw (1.5,1.5) node {$4$}; 
\draw (1.5,2.5) node {$3$}; 
\end{scope}
\begin{scope}[shift={(15,-1)}]
\draw (1,2.5) node {$t^{(0-2,2,0,-4,-2)}$}; 
\end{scope}

\begin{scope}[shift={(20,2)}]
\draw [thick] (0,0) rectangle (2,3);
\draw[step=1] (0,0) grid (2,3);
\draw (0.5,0.5) node {$5$};
\draw (0.5,1.5) node {$3$}; 
\draw (0.5,2.5) node  {$1$}; 
\draw (1.5,0.5) node {$6$}; 
\draw (1.5,1.5) node {$4$}; 
\draw (1.5,2.5) node {$2$}; 
\end{scope}
\begin{scope}[shift={(20,-1)}]
\draw (1,2.5) node {$t^{(0,2,-2,0,-4,-2)}$}; 
\end{scope}

\end{tikzpicture}

\end{example}

Note that for all $R \in \sytk$, box $\squarebox{1}$
is always in the upper left corner, so $\diag_R(1)=0$
and $\squarebox{n}$ is always in the lower right corner, so $\diag_R(n)=k-N$.

\begin{proposition}\label{prop-AHARect} We have:
\begin{enumerate}
\item
The irreducible $H(\Y)$-module $\rect$ has a $\Y$-weight basis,
$$\{ v_R \mid R \in \sytk\}$$
when $t$ is generic, such that each $v_R$ has weight $\wtR(R)$, i.e,
$$Y_i v_R = t^{2 \diag_R(i)}v_R.$$
In particular the central element $\prod_{i=1}^n Y_i$ acts as the scalar
$t^{n(k-N)}$. 
\item
Set $\Yprod = \Sq^{1-n} t^{k-N}$.  Then $\rect$ descends to a representation of $H(\cZ)$, such that
$$Z_i v_R =  t^{2 \diag_R(i)}v_R.$$
\end{enumerate}
\end{proposition}
In Section \ref{sec-functor} we will consider $\HS$, with the specialization $t=\Sq^N=\q$. In light of Proposition \ref{prop-AHARect}, we will therefore further specialize $\Yprod=\Sq^{1-n}t^{(k-N)}=\Sq^{1-N^2}=\q^{1/N-N}$ in the definition of $\HH_{\Sq,t}(\SL_n)$.

\subsection{Induction of the rectangular representation to the DAHA for $\GL$}
\label{sec-induced-moduleGL}
In this section, we prove that when we induce $\rect$ from the affine Hecke algebra to $\HG$ under the specialization
$q=t^{-2k}$ it has unique simple quotient.

Let
$$\MkN = \Ind_{H(\Y)}^{\HH_{q,t}(\GL_n)} \rect$$ denote the induction of $\rect$ from $H(\Y)$ to $\HH_{q,t}(\GL_n)$.
This module has basis
$$\{ T_\sigma \otimes v_R \mid \sigma \in \AffSym/\Sn, R \in \syt(k^N) \}$$
which can be ordered (refining weak Bruhat order) so that with respect to this basis $\Y$ acts triangularly.
$\MkN$ thus has  support 
\begin{equation}\label{supp-eqn}
\supp\left(\MkN\right) = \{ \sigma \cdot \wtR(R) \mid \sigma \in \AffSym/\Sn, R \in \syt(k^N)\}.
\end{equation}

\begin{theorem}
\label{thm-usqGL}
Let $q = t^{-2k}$. Then $\MkN$ has unique simple quotient.
\end{theorem}

\begin{proof}
Suppose that $L$ is simple and we have a nonzero map 
$$\MkN =  \Ind_{H(\Y)}^{\HH_{q,t}(\GL_n)} \rect \to L.$$
By adjointness of induction and restriction
(Frobenius reciprocity)
$L$ contains a weight vector of weight $\wtR(R)$ for any
fixed $R \in \syt(k^N)$,
and any such vector generates $L$ as an $\HH_{q,t}(\GL_n)$-module.
We shall show that the $\wtR(R)$-weight space of $\MkN$ is
one-dimensional, proving $\MkN$ can only have one simple quotient.

Since $\rect$ is $\Y$-semisimple, it suffices to determine which
$\sigma \in \AffSym / \Sn$ stabilize $\wtR(R)$;
the size of the stabilizer is the dimension of the weight space.  On the one hand, the weights of $\MkN$ are given in \eqref{supp-eqn}.  On the other hand, each double coset $\Sn \backslash \AffSym / \Sn$
contains ``translation by a dominant weight"---that is to say,
any weight in the support of $\MkN$ has the form
\begin{equation}\tau\cdot  
(q^{\gamma_1} t^{2a_1}, \ldots, q^{\gamma_n} t^{2a_n}) = \tau\cdot
( t^{2a_1- 2k \gamma_1}, \ldots,  t^{2a_n- 2k \gamma_n}),\label{mu-eqn}
\end{equation}
where $\tau \in \Sn$, $\gamma_i \in \ZZ$, $\gamma_1 \le \gamma_2 \le \cdots \le \gamma_n$,
and $t^{(2a_1, \ldots, 2a_n)} = \wtR(R)$ for some $R \in \syt(k^N)$.  All such $R$ satisfy:
\begin{equation} a_1 = 0, \qquad 1-N \le a_i \le k-1\quad
(i=2,\ldots n), \qquad a_n =
k-N;\label{content-bounds-equation}
\end{equation}
see Figure
\ref{content-bounds-figure}.

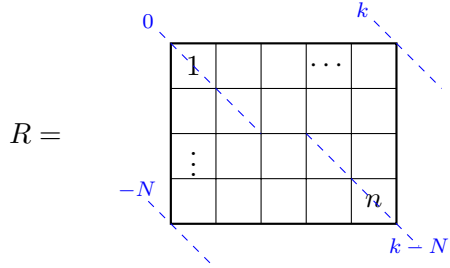
\begin{figure}[h]
 \begin{tikzpicture}[scale=0.6]
\begin{scope}[shift={(15,1)}]
\draw [thick] (0,0) rectangle (5,4);
\draw[step=1] (0,0) grid (5,4);
\draw (-3,2) node {$R=$};
\draw (0.5,1.5) node {$\vdots$};
\draw (0.5,3.5) node  {$1$};
\draw (3.5,3.5) node  {$\cdots$};
\draw (4.5,0.5) node  {$n$};
\draw [dashed,blue] (-.25,4.25) -- (2, 2); 
\draw [dashed,blue] (4.5,4.5) -- (6, 3);   
\draw [dashed,blue] (-.5,.5) -- (1, -1);   
\draw [dashed,blue] (3,2) -- (5.5, -.5);  
\draw (-0.5,4.5) node[blue] {\scriptsize $0$};
\draw (4.25,4.75) node[blue] {\scriptsize $k$};
\draw (-0.75,0.75) node[blue] {\scriptsize $-N$};
\draw (5.5,-0.5) node[blue] {\scriptsize $k-N$};
\end{scope}
\end{tikzpicture}
\caption{The contents of the corners of the rectangle imply the restrictions on $a_i=\diag_R(i)$ in equation \eqref{content-bounds-equation}.}
\label{content-bounds-figure}
\end{figure}

We now claim that the only possible $\gamma$ which can appear in equation \eqref{mu-eqn} is the trivial one,
$$\gamma_1  = \gamma_2 = \cdots = \gamma_n = 0.$$

Since $\gamma$ is dominant, it suffices to show that $\gamma_1\geq 0$ and $\gamma_n\leq 0$.  Indeed, suppose for the sake of contradiction that $\gamma_1 < 0$, then we have $a_1-k \gamma_1 \ge k > k-1$, hence  $a_1-k \gamma_1 \neq a_i$ for any $i$, by the restrictions \eqref{content-bounds-equation}.  Likewise if we suppose $\gamma_n>0$, we have $a_n-k \gamma_n \le -N < 1-N$ so $a_n-k \gamma_n \neq a_i$ for any $i$, again by \eqref{content-bounds-equation}.

Given that $\sigma \in \AffSym/\Sn$ is taken to be of minimal length, this implies $\sigma = \id$. 
In other words, for any $R \in \syt(k^N)$ the $\wtR(R)$ weight space of $\MkN$ is one-dimensional, and it follows that there is a unique simple quotient.

\end{proof}

In the next section we will explicitly construct this unique simple
quotient and thereby show it is $\Y$-semisimple. 

\subsection{The periodic rectangle and the double affine Hecke algebra}
\label{sec-periodic}
The construction of irreducible $\Y$-semisimple
$\HH_{q,t}(\GL_n)$-modules starts with the notion of a periodic skew
diagram \cite{Cherednik03}, and the subsequent notion of a
periodic skew tableau.  Let us recall the construction
in \cite{SV} here, modified to conform to the conventions
of this paper.

Given a skew Young diagram $\mu/\lambda$, we make it periodic as follows.  We embed the $n$ cells of a skew diagram $\mu/\lambda$
into $\ZZ^2$ using matrix style coordinates $(\mathrm{row},\mathrm{column})$.%
\footnote{This is a standard choice for Young diagrams in English notation, in place of Cartesian coordinates.}

 For each $r$, we have the $r$-shifted diagram,
$$(\mu/\lambda)[r] = \mu/\lambda + r(-\ell,k)$$
where $k = \mu_1=$ the number of columns of $\mu$,
and $\ell$ is determined by $t=q^{-2(k + \ell)}$. 
The condition that $\mu/\lambda[0] \cup \mu/\lambda[1]$ again forms a skew diagram forces
\begin{equation}\label{skew-condition}
l(\mu) - \mathrm{mult}(\mu_1) \le \ell,\qquad l(\lambda) \le \ell,
\end{equation}
where $l$ denotes the number of rows of the diagram.
The case we consider in this paper is very special: we have $\lambda =  \emptyset$ and $\mu = (k^N)$,
so the first condition in \eqref{skew-condition} is vacuous, as
$l(\mu) - \mathrm{mult}(\mu_1)  = N - N = 0$.
Further, for $\HG$ we specialize $q=t^{-2k}$ and so $\ell=0$.
In other words
$$(k^N) [r] = (k^N) + r(0, k)$$
so the boxes of the periodic (skew) diagram form an $N \times \infty$ strip, see Figure \ref{N-strip}.

\begin{figure}[h]
 \begin{tikzpicture}[scale=0.6]
\begin{scope}[shift={(5,2)}]
\draw [very thick] (0,0) rectangle (2,3);
\draw[step=1, very thick] (0,0) grid (2,3);
\draw (0.9,3.5) node  {$\mu[0]$};

\draw  (-2,0) rectangle (0,3);
\draw[step=1,gray] (-2,0) grid (0,3);
\draw (-1.2,3.5) node  {$\mu[-1]$};

\draw  (2,0) rectangle (4,3);
\draw[step=1,gray] (2,0) grid (4,3);
\draw (2.9,3.5) node  {$\mu[1]$};

\draw  (4,0) rectangle (6,3);
\draw[step=1,gray] (4,0) grid (6,3);
\draw (4.9,3.5) node  {$\mu[2]$};

\draw[thick] (-5,3)--(9,3);
\draw[thick] (-5,0)--(9,0);
\end{scope}
\end{tikzpicture}
\caption{The diagram $\mu = (2^3)$ is made periodic by shifting horizontally.}\label{N-strip}
\end{figure}
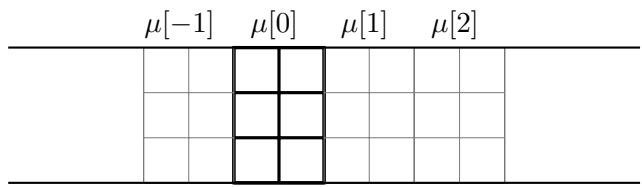

In the $\GL_n$ case, we always consider the fundamental domain $\mu[0]$
to be anchored on the $0$-diagonal. 
In other words, the $(1,1)$ cell is always the upper left corner of
$\mu[0]$. 

\begin{definition}\label{def-psytk}  Let $n=kN$.  An $n$-periodic standard tableaux of shape $\mu=(k^N)$
is a bijection $R:\ZZ \to \{$boxes of $N \times \infty$ strip$\}$
such that:
\begin{itemize}
\item fillings increase across rows and down columns,
\item
the fillings of $\mu[0]$ are distinct $\bmod$  $n$,
\item
the fillings of $\mu[r]$ are those of $\mu[0]$ $+ nr$.
\end{itemize}
We will denote the set of all such tableaux $\psytk$.
\end{definition}

An $R\in\psytk$ is completely determined by the fillings of $\mu[0]$, see Figure \ref{periodic-ex}.  However it may happen that the filling of $\mu[0]$ is row- and column-increasing, but its periodization is not standard, see Figure \ref{periodic-non-ex}.

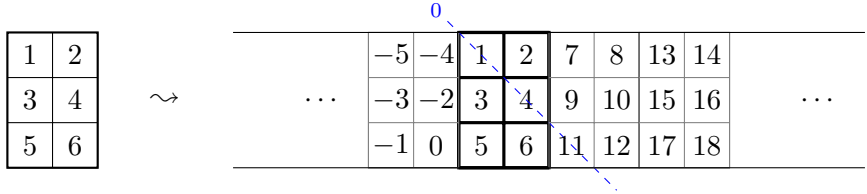
\begin{figure}[h]
\begin{tikzpicture}[scale=0.6]
\begin{scope}[shift={(-3,2)}]
\draw [thick] (0,0) rectangle (2,3);
\draw[step=1] (0,0) grid (2,3);
\draw (0.5,0.5) node {$5$};
\draw (0.5,1.5) node {$3$}; 
\draw (0.5,2.5) node  {$1$}; 
\draw (1.5,0.5) node {$6$}; 
\draw (1.5,1.5) node {$4$}; 
\draw (1.5,2.5) node {$2$};
\draw (3.5,1.5) node {$\leadsto$};
\end{scope}

\begin{scope}[shift={(7,2)}]
\draw [very thick] (0,0) rectangle (2,3);
\draw[step=1, very thick] (0,0) grid (2,3);
\draw (0.5,0.5) node {$5$};
\draw (0.5,1.5) node {$3$}; 
\draw (0.5,2.5) node  {$1$}; 
\draw (1.5,0.5) node {$6$}; 
\draw (1.5,1.5) node {$4$}; 
\draw (1.5,2.5) node {$2$}; 
\draw [dashed,blue] (-0.25,3.25) -- (3.5, -0.5);
\draw (-0.5,3.5) node[blue] {\scriptsize $0$};

\draw  (-2,0) rectangle (0,3);
\draw[step=1,gray] (-2,0) grid (0,3);
\draw (-0.5,0.5) node {$0$}; 
\draw (-0.5,1.5) node {$-2$}; 
\draw (-0.5,2.5) node {$-4$};
\draw (-1.5,0.5) node {$-1$};
\draw (-1.5,1.5) node {$-3$}; 
\draw (-1.5,2.5) node  {$-5$};

\draw  (2,0) rectangle (4,3);
\draw[step=1,gray] (2,0) grid (4,3);
\draw (2.5,0.5) node {$11$};
\draw (2.5,1.5) node {$9$}; 
\draw (2.5,2.5) node  {$7$}; 
\draw (3.5,0.5) node {$12$}; 
\draw (3.5,1.5) node {$10$}; 
\draw (3.5,2.5) node {$8$}; 

\draw  (4,0) rectangle (6,3);
\draw[step=1,gray] (4,0) grid (6,3);
\draw (4.5,0.5) node {$17$};
\draw (4.5,1.5) node {$15$}; 
\draw (4.5,2.5) node  {$13$}; 
\draw (5.5,0.5) node {$18$}; 
\draw (5.5,1.5) node {$16$}; 
\draw (5.5,2.5) node {$14$}; 

\draw (-3,1.5) node {$\cdots$}; 
\draw (8,1.5) node {$\cdots$}; 
\draw (-5,3)--(9,3);
\draw (-5,0)--(9,0);
\end{scope}
\end{tikzpicture}
\caption{The filling of $\mu[0]$ completely determines the filling of $R \in \psytn{6}(2^3)$, via the periodicity constraint.}\label{periodic-ex}
\end{figure}

\begin{figure}[h]
 \begin{tikzpicture}[scale=0.6]

\begin{scope}[shift={(-3,2)}]
\draw [thick] (0,0) rectangle (2,3);
\draw[step=1] (0,0) grid (2,3);
\draw (0.5,0.5) node {$5$};
\draw (0.5,1.5) node {$3$}; 
\draw (0.5,2.5) node  {$1$}; 
\draw (1.5,0.5) node {$12$}; 
\draw (1.5,1.5) node {$10$}; 
\draw (1.5,2.5) node {$8$}; 
\draw (3.5,1.5) node {$\leadsto$};
\end{scope}

\begin{scope}[shift={(7,2)}]
\draw [very thick] (0,0) rectangle (2,3);
\draw[step=1, very thick] (0,0) grid (2,3);
\draw (0.5,0.5) node {$5$};
\draw (0.5,1.5) node {$3$}; 
\draw (0.5,2.5) node  {$1$}; 
\draw (1.5,0.5) node {$12$}; 
\draw (1.5,1.5) node {$10$}; 
\draw (1.5,2.5) node {$8$}; 
\draw [dashed,blue] (-0.25,3.25) -- (3.5, -0.5);
\draw (-0.5,3.5) node[blue] {\scriptsize $0$};

\draw  (-2,0) rectangle (0,3);
\draw[step=1,gray] (-2,0) grid (0,3);
\draw (-0.5,0.5) node {$6$}; 
\draw (-0.5,1.5) node {$4$}; 
\draw (-0.5,2.5) node {$2$}; 
\draw (-1.5,0.5) node {$-1$};
\draw (-1.5,1.5) node {$-3$}; 
\draw (-1.5,2.5) node  {$-5$}; 

\draw  (2,0) rectangle (4,3);
\draw[step=1,gray] (2,0) grid (4,3);
\draw (2.5,0.5) node {$11$};
\draw (2.5,1.5) node {$9$}; 
\draw (2.5,2.5) node  {$7$}; 
\draw (3.5,0.5) node {$18$}; 
\draw (3.5,1.5) node {$16$}; 
\draw (3.5,2.5) node {$14$}; 

\draw  (4,0) rectangle (6,3);
\draw[step=1,gray] (4,0) grid (6,3);
\draw (4.5,0.5) node {$17$};
\draw (4.5,1.5) node {$15$}; 
\draw (4.5,2.5) node  {$13$}; 
\draw (5.5,0.5) node {$24$}; 
\draw (5.5,1.5) node {$22$}; 
\draw (5.5,2.5) node {$20$}; 

\draw (-3,1.5) node {$\cdots$}; 
\draw (8,1.5) node {$\cdots$}; 
\draw (-5,3)--(9,3);
\draw (-5,0)--(9,0);
\end{scope}

\end{tikzpicture}
\caption{The filling of $\mu[0]$ is row- and column-increasing, but its periodization is not standard.}\label{periodic-non-ex}
\end{figure}
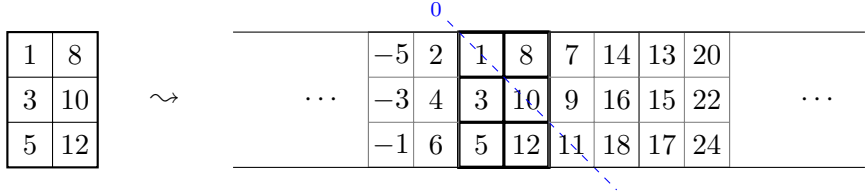

The diagonal and weight functions are defined for periodic skew tableaux
$R \in \psytk$ similarly to those for standard Young tableaux, in Definitions
\ref{def-diagRect} and \ref{def-weightRect}.  The diagonal function is
$\diag_R(i) = m-j$, when $\squarebox{i}$ lies in row $j$ and column $m$.
Note this is defined for all $i \in \ZZ$ and satisfies
$\diag_R(i+n) = \diag_R(i) + k$.  The
weight, $wt(R)$, of $R$ is the tuple $(t^{2 \diag_R(1)}, t^{2 \diag_R(2)},
\ldots, t^{2 \diag_R(n)})$.

The $\AffSym$ action on $\ZZ$ descends to an action on periodic
tableaux, as follows.  We set $\sigma \cdot R$ to be the tableau where
$\squarebox{\scriptstyle{i}}$ is replaced with  $\squarebox{\scriptstyle{\sigma(i)}}$.

The function $\diag_R$ is compatible with this action:
$$\diag_{\sigma \cdot R}( \sigma(i) ) = \diag_R( i)$$
for any $\sigma \in \AffSym$.
Furthermore the action intertwines the action of $\AffSym$ on $(\K^\times)^n$
described in \eqref{AffSymmOnGLWts}: we have $\wtR(\sigma \cdot R) = \sigma \cdot \wtR(R).$
We note that $\sigma \cdot R$ need not be standard, even if $R$ is.  
\begin{remark}\label{rmk-shift} We note that any domain for the $n$-periodicity in Definition \ref{def-psytk} is also a domain for the $\pi^n$-action.  Note that $\pi^n$ shifts the $N\times \infty$ strip $k$ steps horizontally.
\end{remark}

\begin{theorem}[\cite{Cherednik03}, \cite{SV}]\label{thm-dahaLkN}
Let $L(k^N)$ denote the linear span over $\K$ of
$$\{v_R \mid R\in \psytk\}.$$ 
When $q=t^{-2k}$
 there exists a unique irreducible representation of $\HG$ on $L(k^N)$, such that each $v_R$ is a $\Y$-weight vector of weight $\wtR(R)$, i.e.
$$Y_iv_R = t^{2\diag_R(i)}v_R, \quad(i=1,\ldots, n).$$
\end{theorem}

\begin{corollary}
\label{cor-LkN}
When $q=t^{-2k}$,
the unique simple quotient of $\MkN$ is $\LkN$; in particular
it is $\Y$-semisimple and its support is given in Theorem \ref{thm-dahaLkN}.
\end{corollary}
\begin{proof} It follows by Frobenius reciprocity that there exists a non-zero map from $M(k^N)$ to $L(k^N)$, hence $L(k^N)$ is indeed the unique simple quotient guaranteed by Theorem \ref{thm-usqGL}.
\end{proof}

Both Theorem \ref{thm-dahaLkN} and Corollary \ref{cor-LkN}
justify associating the irreducible $\HG$-module with the
$N \times k$ rectangle.

\subsection{$\SL$ modifications}
\label{sec-SLpsyt}
When considering $\SL$ versus $\GL$, we need to make
the following modifications both to Theorems \ref{thm-usqGL} and \ref{thm-dahaLkN}
and the underlying combinatorics.
Further we specialize $t=\Sq^N$ and $\Yprod=\Sq^{1-N^2}$.
 They are consistent with
the modified $\AffSym$ action on $(\K^\times)^n$
described in \eqref{AffSymmOnSLWts}.

The role of periodic skew diagram $\psytk$ in the $GL$ case is now played by the set $\SLpsyt$ of equivalence classes: we impose
the equivalence relation $R \prel \pi^n \cdot R$, and denote equivalence classes as $\bR$.
For $\bR \in \SLpsyt$ we modify the function $\diag_R$ to
$\diag_{\bR}$ as follows.

When a fundamental domain such as $\mu[0]$ is filled with any
$R_0 \in \sytk$, it has ``filling sum,"
$$\fsum = \sum_{i=1}^n i.$$
For $R \in \psytk$, any of its fundamental domains (under $\pi^n$-shifts, see Remark \ref{rmk-shift})
 has filling sum $\fsum +np$
for some $p \in \ZZ$ (even if the domain is not of the form $\mu[r]$).
For any $N \times k $ rectangle of $\bR$ we label with $\frac pN$ the diagonal
 through its northwest corner; we will call this its {\emph{\NW diagonal}}.

To see that this gives a well-defined diagonal labelling to all of $R$, we note that if a domain has filling sum
$\fsum+np$, then the domain that is shifted one unit right has filling sum
$\fsum + np + nN$ and $\frac{p+N}{N} = \frac pN +1$.
More precisely, if $\squarebox{\scriptstyle{i}}$  sits in the
domain $\mu[r]$ in the  $j$th row and $(m + rk)$th column, 
\begin{gather}
\label{eq-pSLdiagi}
\diag_{\bR}(i) = m-j + \frac pN
\end{gather}
where $p$ is determined as above with respect to $\mu[r]$. In particular $\diag_{\bR}(i)$ is
independent of $r$, and $\diag_{\bR}(i) = \diag_{\pi^n \cdot \bR}(i)$
since $\squarebox{\scriptstyle{i}}$ now sits in domain $\mu[r+1]$
of  $\pi^n \cdot \bR$ but the local information of $p$
stays the same. 
The function $\diag_{\bR}$ is defined for all $i \in \ZZ$ and satisfies
$\diag_{\bR}(i+n) = \diag_{\bR}(i) + k$. 

The weight of $\bR \in \SLpsyt$ is the tuple $(t^{2 \diag_\bR(1)},
t^{2 \diag_\bR(2)}, \ldots, t^{2 \diag_\bR(n)})$.

\begin{example}
See Figure \ref{fig-sl-diagonals-periodic} for an example of how
the diagonal labels change within a $\pi$ orbit.
One should also compare this to Figure \ref{fig-sl3-diagonals-skew}, where only the locations of filling by \{1,2,3\} are marked.
We let $N=3, k=1$, so $\fsum = 6$.
In the first periodic tableau $R$, $1+5+6 = \fsum + 3(2)$
so the chosen fundamental domain has \NW diagonal labeled
$\frac 23$.  In the second, $2+6+7 = \fsum  + 3(3)$, and so on.
In the fourth periodic tableau the chosen domain has diagonal
labeled $\frac 53$
and passes through 
$\squarebox{\scriptstyle{4}}$,
but the diagonal one step left has label
$\frac 23$
and passes through 
$\squarebox{\scriptstyle{1}}$;
and indeed $R \prel \pi^3 \cdot R$.
In other words, 
$\diag_{\bR}(1) = \frac 23
= \diag_{\overline{\pi^3 \cdot R}}(1).$

We also list $\wtR(\bR)$ in Figure \ref{fig-sl-diagonals-periodic}.  Note that
$$\pi \cdot (t^{4/3}, t^{-8/3},
t^{-14/3}) = (\Sq^{-2(3)+2} t^{-14/3}, \Sq^2 t^{4/3}, \Sq^2 t^{-8/3}) 
= (t^{-6}, t^{2}, t^{-2}),$$
as $t=\Sq^3$.

\begin{figure}[h]
\begin{tikzpicture}[scale=0.6]
\begin{scope}[shift={(-3,2)}]
\draw [thick] (0,0) rectangle (1,3);
\draw[step=1] (0,0) grid (1,3);
\draw (0.5,0.5) node {$6$};
\draw (0.5,1.5) node {$5$};
\draw (0.5,2.5) node  {$1$};
\draw [dashed,blue] (-0.25,3.25) -- (1.5, 1.5);
\draw (-0.5,3.5) node[blue] {$\frac 23$};
\draw (0.5,-.5) node  {$\bR$};
\draw (0.5,-1.7) node  {$t^{(4/3,-8/3,-14/3)} $}; 
\end{scope}

\begin{scope}[shift={(2,2)}]
\draw [thick] (0,0) rectangle (1,3);
\draw[step=1] (0,0) grid (1,3);
\draw (0.5,0.5) node {$7$};
\draw (0.5,1.5) node {$6$};
\draw (0.5,2.5) node  {$2$};
\draw [dashed,blue] (-0.25,3.25) -- (1.5, 1.5);
\draw (-0.5,3.5) node[blue] {$1$};
\draw (0.5,-.5) node  {$\pi\cdot\bR$};
\draw (0.5,-1.7) node  {$t^{(-6,2,-2)} $}; 
\end{scope}

\begin{scope}[shift={(7,2)}]
\draw [thick] (0,0) rectangle (1,3);
\draw[step=1] (0,0) grid (1,3);
\draw (0.5,0.5) node {$8$};
\draw (0.5,1.5) node {$7$};
\draw (0.5,2.5) node  {$3$};
\draw [dashed,blue] (-0.25,3.25) -- (1.5, 1.5);
\draw (-0.5,3.5) node[blue] {$\frac 43$};
\draw (0.5,-.5) node  {$\pi^2\cdot\bR$};
\draw (0.5,-1.7) node  {$t^{(-10/3,-16/3,8/3)} $}; 
\end{scope}

\begin{scope}[shift={(14,2)}]
\draw [thick] (0,0) rectangle (1,3);
\draw[step=1] (0,0) grid (1,3);
\draw[gray,step=1] (-2,0) grid (0,3);
\draw (0.5,0.5) node {$9$};
\draw (0.5,1.5) node {$8$};
\draw (0.5,2.5) node  {$4$};
\draw[gray] (-0.5,0.5) node {$6$};
\draw[gray] (-0.5,1.5) node {$5$};
\draw[gray] (-0.5,2.5) node  {$1$};
\draw[gray] (-1.5,0.5) node {$3$};
\draw[gray] (-1.5,1.5) node {$2$};
\draw[gray] (-1.5,2.5) node  {$-2$};
\draw[gray] (1.5,2.5) node {$7$};
\draw [dashed,blue] (-0.25,3.25) -- (1.5, 1.5);
\draw (-0.5,3.5) node[blue] {$\frac 53$};
\draw [dashed,blue] (-1.25,3.25) -- (0.5, 1.5);
\draw (-1.5,3.5) node[blue] {$\frac 23$};
\draw (0.5,-.5) node  {$\pi^3\cdot\bR$};
\draw (0.5,-1.7) node  {$t^{(4/3,-8/3,-14/3)} $}; 
\end{scope}

\end{tikzpicture}
\caption{The label of the \NW diagonal of a fundamental 
rectangle depends on its filling sum in type $\SL$. }
\label{fig-sl-diagonals-periodic}
\end{figure}
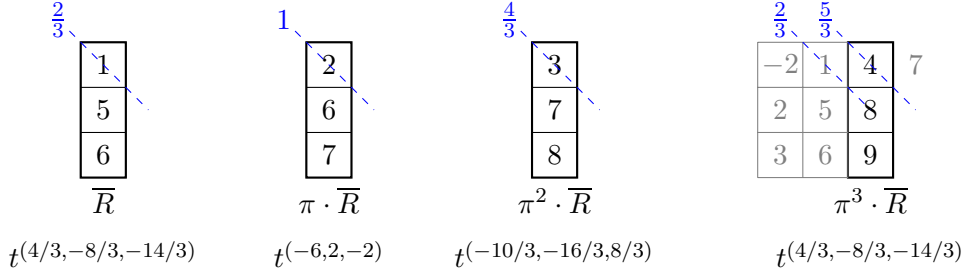
\end{example}

Just as in the $\GL_n$ case, $\AffSym$ acts on $\SLpsyt$.
It is no longer true that
$\diag_{\overline{\sigma \cdot R}}(\sigma(i))$
agrees with $\diag_{\bR}(i)$
for $\sigma \in\AffSym$ (in particular for $\sigma=\pi$).
However, we do still have the intertwining property
$$\wtR(\sigma \cdot \bR) = \sigma \cdot \wtR(\bR)$$
using the $\SL_n$ modified action of $\AffSym$ on $(\K^\times)^n$
described in \eqref{AffSymmOnSLWts}.

Let $\sLkN$ denote the linear span over $\K$ of
$$\{v_{\bR} \mid \bR\in \SLpsyt\}.$$ 
\begin{theorem}
\label{thm-sLkN}
When $\Sq^N=t$
and $\Yprod=\Sq^{1-N^2}$
 there exists a unique irreducible representation of $\HS$ on $\sLkN$
such that each $v_{\bR}$ is a $\Z$-weight vector of weight $\wtR(\bR)$, i.e.
$$Z_iv_{\bR} = t^{2\diagRi{\bR}{i}}v_{\bR}, \quad(i=1,\ldots, n),$$
and such that 
$$\pi^N v_{\bR_0} =  v_{\bR_0},$$
where $\bR_0$ is as in Figure \ref{stabilizedR} below.
\end{theorem}

The existence of the module $\sLkN$ has been established above,
constructing it combinatorially
via equivalence classes of standard periodic tableaux. Unlike the $\GL$ case,
$\sLkN$ is not the unique simple quotient of the induced module $\sMkN$. Rather
we have:

\begin{theorem}
\label{thm-sMkN}
Let $\Sq^N=t$
and $\Yprod=\Sq^{1-N^2}$ 
Let $a \in \K^\times$ be a primitive $n$th root of unity.
\begin{enumerate}
\item
\label{item1-sMkN}
The induced module 
 $\sMkN = \Ind_{H(\Z)}^{\HS} \rect$ has maximal semi-simple quotient
$$\sLkN \oplus  \sLkNa{a} \oplus \cdots \oplus \sLkNa{a^{k-1}}.$$

\item
\label{item2-sMkN}
We have $\overline{L}(k^N)^a \cong \overline{L}(k^N)^b$ if and only if
$a^N = b^N$.

\item
\label{item3-sMkN}
These $k$ twists are all the representations of $\HS$ that have the
same support as $\sLkN$. Their isomorphism type is distinguished
by the extra data of
$\pi^N v_{\bR_0} = a^N  v_{\bR_0}.$
\end{enumerate}
\end{theorem}

\begin{proof}
Along the lines of Theorem \ref{thm-usqGL}, one sees that for $\bR_0$ depicted in Figure \ref{stabilizedR},
$\wtR(\bR_0)$ has stabilizer  $\langle \pi^N \rangle$
via the action \eqref{AffSymmOnSLWts},
and hence the $\wtR(\bR_0)$-weight space of $\sMkN$ has dimension $k=n/N$.


\begin{figure}[h]
 \begin{tikzpicture}[scale=0.8]
\draw [thick] (0,0) rectangle (6,4);
\draw[step=1] (0,0) grid (6,4);
\draw (-1.5,2) node {$\bR_0 \quad =$};
\draw (0.5,1.5) node {$\vdots$};
\draw (0.5,3.5) node  {\scriptsize$1$}; 
\draw (0.5,2.5) node  {\scriptsize$2$};
\draw (0.5,0.5) node  {\scriptsize$N$};
\draw (1.5,3.5) node  {\scriptsize$N\!+\!1$};
\draw (1.5,0.5) node  {\scriptsize$2N$};
\draw (3.5,3.5) node  {$\cdots$};
\draw (5.5,1.5) node  {\scriptsize$n\!-\!1$};
\draw (5.5,0.5) node  {\scriptsize$n$};
\end{tikzpicture}
\caption{A periodic skew tableau $\bR_0$ stabilized by $\pi^N$.}\label{stabilizedR}
\end{figure}

Let us denote by $\va{a}$ a basis vector of
the one-dimensional $\wtR(\bR_0)$ weight  space
of $(\sLkN)^a$.
Note that $\pi^N$ preserves the weight space, hence it scales
$\va{a}$ by some scalar.
Since any homomorphism $\varphi: (\sLkN)^a \to (\sLkN)^b$ commutes with the
action of $Z_i$, it sends $\va{a}$ to some multiple of $\va{b}$, say
$\varphi(\va{a}) = c \va{b}$.  Then
$$a^N c \va{a} = \varphi(a^N \va{a}) = \varphi(\pi^N \va{a}) = \pi^N \varphi(\va{a}) = c b^N \va{b}.$$
If $\varphi \neq 0$ then $a^N = b^N$.
This proves the ``only if" part of \eqref{item2-sMkN}.

As in the $GL$ case, Frobenius reciprocity gives a map,
$$\sMkN = \Ind_{H(\Z)}^{\HS} \rect \to \sLkNa{a},$$
as the restriction of $\sLkNa{a}$ to $H(\Z)$ contains $\rect$.
 As each of these quotients has a one-dimensional $\wtR(\bR_0)$-weight space,
and $\sMkN$ has $k$-dimensional $\wtR(\bR_0)$-weight space,
$\sMkN$ can have at most $k$ simple quotients.
Now using the pigeon-hole principle, we complete the proof of
\eqref{item2-sMkN}, as well as \eqref{item1-sMkN}.
In particular the discussion on $\varphi$ shows \eqref{item3-sMkN} as well.

\end{proof}

\begin{remark}\label{rem-SLsub}Alternatively, we could have replaced $\HS$ by its subalgebra generated by $\langle T_0, T_1,
\ldots, T_{n-1}, Z_1^{\pm 1} \rangle$,
so that the weight $\wtR(\bR_0)$ (and indeed the weight of any $\bR\in\sytk$) would once again have trivial stabilizer.  Hence, defining $\sMkN$ by induction in the same way would yield a module with unique simple quotient $\sLkN$.  We note that $\HS$ is a free module of rank $n$ over this subalgebra
with basis $\{1, \pi, \ldots, \pi^{n-1} \}$.  This subalgebra is in some ways more similar to the RCA of type $\SL$.
\end{remark}

\section{The functor, and the isomorphism}
\label{sec-main-results}

\subsection{The functor $F_n^G$}
\label{sec-functor}

Let $G=SL_N$ or $GL_N$, and let $M$ be a $\cD_\q(G)$-module.  Let $n\in\mathbb{N}$, and consider the space,
\begin{align*}F^{SL_N}_n(M) &= \left(\underset{n}{V}\ot \cdots \ot \underset{1}{V}\ot M\right)^{U_\q(\slN)},\quad\textrm{and}\\
F^{GL_N}_n(M) &= \left(\detq^{-k}(V)\ot\underset{n}{V}\ot \cdots \ot \underset{1}{V}\ot M\right)^{U_\q(\glN)},
\end{align*}
of invariants, respectively $\detq^{k}$-variants, in the tensor product of $M$ with the $n$-fold tensor power of the defining  $N$-dimensional representation,
$V = V_{\ep{1}}$
 of $U_\q(\g)$.  In order to match various conventions (while breaking others), we index the tensor factors from right to left, as indicated in the subscripts above.  We note that $-k$ is the unique possible tensor exponent of $\detq(V)$ such that $F^{GL_N}_n(\cO_\q(GL_N))$ is non-zero, and that similarly $n$ must be an integer multiple of $N$ in the $SL$ case, see Remark \ref{rmk-kint}.

In \cite{Jordan2008}, an action of the double affine Hecke algebra was constructed on the space $F^{SL_N}_n$.  Let us recall the construction here, and formulate its $GL$-modification.

\begin{theorem}[\cite{Jordan2008}] Let $G=SL_N$, or $GL_N$, respectively.  Let $M$ be a module for $\cD_\q(G)$, let $k$ be a positive integer, and let $n=kN$.  There is a unique representation of $B_n^{Ell}$ (resp. $B_{n,1}^{Ell}$) on $F^G_n(M)$ such that:
\begin{enumerate}
\item
Each generator $\brT_i$ $(i=1, \ldots, n-1)$
 acts by the braiding $\sigma_{V,V}$
on the $\underset{i+1}{V}\ot\underset{i}{V}$ factors.
\item
For $\GLN$, $\brT_n^2$
acts by the double braiding on $\detq^{-k}(V) \ot
\underset{n}{V}$.
\item The operator $\brY_1$ acts only in the rightmost two tensor factors, $\underset{1}{V}\ot M$, via
$$\brY_1=\sigma_{M_\lhd,V}\circ\sigma_{V,M_\lhd}, \textrm{ (the double-braiding of $V$ and $M$, using $\partial_\lhd$)}.$$
\item The operator $\brX_1$ acts only in the rightmost two tensor factors $\underset{1}{V}\ot M$, via
$$\brX_1 = V\ot M \xrightarrow{\Delta_{V}\ot \id_{M}} V\ot \cO_\q(G)\ot M \xrightarrow{\id_{V}\ot \act_{M}} V\ot M,$$
where in the second arrow, $\cO_\q(G)$ acts on $M$ via the homomorphism $\fun:\cO_\q(G)\to\cD_\q(G)$.
\end{enumerate}
\end{theorem}

\begin{remark} Since $B_n^{Ell}$ and  $B_{n,1}^{Ell}$ are generated
by the
operators listed above,
their  action is uniquely determined.  That these operators satisfy the defining relations listed in Proposition \ref{prop-braid} is established in \cite[Theorem 22 and Corollary 24]{Jordan2008}.
\end{remark}

\begin{corollary}
The $B_n^{Ell}$-action on $F_n^{SL_N}(M)$ is natural in $M$ and $V$, satisfies
the additional relations
$$(\brT_i-\q^{-1/N}\q)(\brT_i+\q^{-1/N}\q^{-1})=0, \qquad
\brY_1\cdots \brY_n=\q^{n(1/N-N)}.$$

Hence by Proposition \ref{SLnBraidtoDAHA} we have an exact functor,
$$F^{SL_N}_n: \cD_\q(SL_N)\modu\to \HH_{\Sq,t}(SL_n)\modu,$$
for $\Sq=\q^{1/N}$, $t=\q$, and $\Yprod = \q^{1/N - N}$.
\end{corollary}

\begin{proof} In \cite{Jordan2008}, Proposition 30, it is proved that $\brY_1\cdots \brY_n$ acts as $(\nu^{-1}|_V)^n$.  We compute, using \eqref{ribbon-elt}, that
$$\nu^{-1}|_V =
\q^{-\langle\e{1} +2\rho,\e{1}\rangle}
= \q^{1/N-N},$$
as desired.\end{proof}

\begin{corollary}\label{cor-GLDAHA}
The $B_{n,1}^{Ell}$-action on $F_n^{GL_N}(M)$ is natural in $M$ and $V$, satisfies the additional relations
\begin{equation}\label{key-gln-relns}(\brT_i-\q)(\brT_i+\q^{-1})=0, \qquad \brT_n^2=\q^{-2k}.\end{equation}
Hence by Proposition \ref{prop-braid-to-DAHA-GL} we have an exact functor,
$$F^{GL_N}_n: \cD_\q(GL_N)\modu\to \HH_{q,t}(GL_n)\modu,$$
for $q=\q^{-2k}$ and $t=\q$.
\end{corollary}

\begin{proof} We outline how to modify the proof from \cite{Jordan2008}, to obtain relations \eqref{key-gln-relns} when $SL_N$ is replaced by $GL_N$. The differing form of the Hecke relation for $\brT_1,\ldots, \brT_{n-1}$ is clear from the quantum $R$-matrix for $GL_N$ compared to $SL_N$.  The quantum determinant representation $\detq(V)$ and all its tensor powers are invertible, which means that $\detq^{-k}(V)\ot V$ is irreducible, so that $\brT_n^2$ must act as a scalar; this scalar can be computed using the ribbon element to be $\q^{-2k}$ as in the proof of Theorem \ref{cor-GLDAHA} above.
\end{proof}

\subsection{The case $M=\cO_\q(G)$}
In the special case $M=\cO_\q(G)$, the action is compatible with Peter-Weyl decomposition, in the following sense.  Let
$$W^n_\lambda = \left\{\begin{array}{ll}(V^{\ot n} \ot V_\lambda^*\ot V_\lambda)^{U_\q(\g)},& G=SL_N\\(\detq^{-k}(V)\ot V^{\ot n} \ot V_\lambda^*\ot V_\lambda)^{U_\q(\g)},& G=GL_N\end{array}\right.,$$
and consider the vector space decomposition,
$$F_n^G(\cO_\q(G)) \cong \bigoplus_{\lambda\in \gsLatp} W^n_\lambda,$$
induced by the Peter-Weyl decomposition, Theorem \ref{thm-Peter-Weyl}.

\begin{remark}\label{rmk-kint} Recall that we have restricted to the case $n=kN$ for a positive integer $k$.  It is elementary to see that otherwise all spaces $W^n_\lambda$, and hence $F_n^G(\cO_\q(G))$, are zero.
\end{remark}

It follows from the definition of $Y_1$ and $T_i$ as braiding operators that they preserve the Peter-Weyl decomposition, and the passage to invariants.  Hence, each subspace $W^n_\lambda\subset F_n^G(\cO_\q(G))$ is a finite-dimensional submodule for the action of the affine Hecke algebra $\cH(\Y)$.

\subsection{Walks, skew diagrams, and a weight basis of the invariants}
In this section we identify the $\cH(\Y)$-modules $W^n_\lambda$ with those constructed in \cite{OR}.

Let $\lambda \in \sLatp$ or $\gLatp$.  To each walk $\underline{u}\in \LWalk Nk{\lambda}$, we associate a unique line $L_{\underline{u}}$ in
$W^n_\lambda$ as follows.  First, in the $SL$ case we have natural isomorphisms,
\begin{align*}W^n_\lambda 
&= \Hom_{U_\q(\slN)}(\mathbf{1},V^{\ot n} \ot V_\lambda^*\ot V_\lambda)\\
&\cong \Hom_{U_\q(\slN)}(\mathbf{1},V^{\ot n} \ot V_\lambda\ot V_\lambda^*)\\
&\cong \Hom_{U_\q(\slN)}\left(V_\lambda,V^{\ot n} \ot V_\lambda \right),\end{align*}
where we have first applied the braiding $\sigma_{V^*, V}$, and then the canonical isomorphism $Hom(\mathbf{1},X\ot Y^*)\cong Hom(Y,X)$.
In the $GL$ case, a similar series of isomorphisms
$$W^n_\lambda 
\cong \Hom_{U_\q(\glN)}(\detq^{k}(V)\ot V_\lambda, V^{\ot n} \ot V_\lambda).$$

The Pieri rule gives a multiplicity-free decomposition,
$$V\ot V_\alpha \cong \bigoplus_{\beta} V_\beta,$$
where each $\beta$ is a dominant weight, which differs from $\alpha$ by an
$\epsilon_i$.
Hence $(V \ot V_\alpha) [\beta] \cong V_\beta$, 
where the notation $X[\lambda]$ denotes the $\lambda$-isotypic component of a
representation $X$.
In particular $\dim \Hom (V_\beta, V \ot V_\alpha) = 1$.
Hence, given a looped walk ${\underline u}$ of length $n$ from $\lambda$ to $\lambda+k\wdet$,
the space
 $$
\Hom(V_{\lambda+k\wdet}, \bigcap_{i=0}^n V^{\ot(n-i)}\ot \left((V^{\ot i}\ot V_\lambda)[u_i]\right))
\subset \Hom(V_{\lambda+k\wdet}, V^{\ot n}\ot V_{\lambda}),$$
is also one-dimensional.  We define $L_{\underline u}$ to be this line. 
In other words,
$L_{\underline u}$
is the subspace of $\Hom(V_{\lambda+k\wdet},V^{\ot
n}\ot V_{\lambda})$, consisting of vectors whose component in each
tensor subfactor $V^{\ot k}\ot V_{\lambda}$ has isomorphism type $u_k$.

 By construction, we have an isomorphism of vector spaces,
$$W^n_\lambda \cong \bigoplus_{\underline{u}\in \LWalk Nk{\lambda}} L_{\underline{u}}.$$

\begin{theorem}\label{thm-Yscalar}
Let $G=GL_N$ or $SL_N$. For any $v\in L_{\uu}$ we have
\begin{equation}\label{Y-scalar}\brY_iv = \q^{
\langle u_i + 2 \rho, u_i \rangle -
\langle u_{i-1} + 2 \rho, u_{i-1} \rangle -
\langle \ep{1}+ 2 \rho, \ep{1} \rangle } v.\end{equation}
\end{theorem}

\begin{proof}
This is essentially the content of \cite[Proposition 3.6]{OR}. 
 We recall the proof for the reader here, in our conventions.  By its construction, $\brY_k=\brT_{k-1}\cdots \brT_1 \brY_1 \brT_1 \cdots \brT_{k-1}$ is the double-braiding of $V$ around $V^{\ot k-1}\ot M$.  Hence, applying equation \eqref{ribbon-id} we have:
$$\brY_k = \id_{V^{\ot n-k}} \ot \left(\Delta^{(k+1)}(\nu)\cdot(\nu^{-1} \ot \Delta^{(k)}(\nu^{-1}))\right)\!\Big|_{V^{\ot k}\ot M}.$$

By definition of the line $L_{\underline{u}}$, and equation \eqref{ribbon-elt}, we have:
\begin{align*}\left(\id_{V^{\ot n-k}} \ot\Delta^{(k+1)}(\nu)\right)\Big|_{L_{\underline{u}}} &= \q^{
\langle u_i + 2 \rho, u_i \rangle},\\
\left(\id_{V^{\ot n-k}} \ot\left(\nu^{-1} \ot \Delta^{(k)}(\nu^{-1})\right)\right)\Big|_{L_{\underline{u}}} &=  \q^{- \langle \ep{1}+ 2 \rho, \ep{1} \rangle- \langle u_{i-1} + 2 \rho, u_{i-1} \rangle }.\end{align*}

Hence we compute that $\brY_i$ scales the line $L_{\underline{u}}$ as claimed.
\end{proof}

\begin{remark}
Recall that the bijection $\Tab$ from Section \ref{sec-walks-to-skew} identifies looped walks at $\lambda$ with skew tableaux in the diagram $\Dlam$.  The diagram $\Dlam$ is natural in light of Theorem \ref{thm-Yscalar}:  tensoring $L_{\uu}$ with $V_{\lambda^*}$ corresponds to filling in the lower right corner of the diagram, thus producing an invariant (in the $SL$ case), or a multiple of the determinant (in the GL case), within $V^{\ot n}\ot V_{\lambda^*}\ot V_{\lambda}$.
\end{remark}

\subsection{From skew tableaux to periodic tableaux}
\label{sec-skew-to-periodic}
While the natural basis of lines coming from Lie theory is indexed by looped walks of length $n$, and hence skew tableaux of size $n$,
 the weight basis for simple and $\Y$-semisimple modules for $\HG$-modules is indexed by periodic tableaux on infinite skew diagrams. 
For simple and $\Z$-semisimple $\HS$-modules the weight basis is indexed
by shift-equivalence classes of periodic tableaux.
  In this section, we construct a bijection between the two bases.

Recall from Section \ref{sec-walks-to-skew} the definition of $\Dlam$ and $\Skk$.  We first observe that $\Dlam$ is also a fundamental domain of the periodization of $\mu=(k^N)$, i.e. of the $N\times \infty$ strip (see Remark \ref{rmk-shift}).  Similarly, ``periodizing" a standard skew tableau in $\cT \in\Skk$ (i.e., filling in the rest of the entries according to the periodicity constraint) yields a well-defined standard periodic tableau in $\psytk$, as soon as we specify the compatibility with the diagonal labelling. 
 See Figure \ref{fig-psiexample}.
We formalize this as follows:

\begin{definition}  The \emph{periodization} maps,
\begin{gather}\label{def:psi}
\Per: \bigsqcup_{\lambda \in \gLatp} \Skk \to \psytk,
\qquad \qquad
\SPer: \bigsqcup_{\lambda \in \sLatp} \Skk \to \SLpsyt,
\end{gather}
send $\cT$ to the unique periodic tableau in $\psytk$ (resp. $\SLpsyt$)
agreeing with $\cT$ in the fundamental domain of shape $\Dlam$ located along the $N\times \infty$ strip so that diagonal labels coincide.
\end{definition}

Note that as the filling of $\cT$ is $\{1, \ldots, n\}$ it is easy to see its periodization is standard.  While it's clear then that $\Per$ is well-defined, we need the following for $\SPer$.

\begin{proposition}
The map $\SPer$ is well-defined.
\end{proposition}
\begin{proof}
Let $\lambda =\sum_i m_i \e{i} \in \sLat$ and consider $\cT \in \Skk$. 
Recall its principal diagonal is labeled $-\frac{|\lambda|}{N}$.
We need to check this agrees with the diagonal labels assigned by
\eqref{eq-pSLdiagi} once we extend $\cT$ periodically to the $N\times \infty$ strip, see Remark \ref{rmk-shift}.

In other words we need to check the $N \times k$ rectangle
whose \NW diagonal agrees with the principal diagonal of $\cT$
has filling sum $\fsum -|\lambda|n$.
Consider the $k$ entries in the first row of $\cT$. 
They lie in $\{1,\ldots,n\}$ and hence contribute to the sum $\fsum$;
they also sit on diagonals labeled $m_1-m_N - \frac{|\lambda|}{N},
\ldots, m_1-m_N+k-1 - \frac{|\lambda|}{N}$. 
Likewise after periodizing, the $k$  entries on diagonals labeled
$$-kr + m_1-m_N - \frac{|\lambda|}{N}, \ldots, -kr + m_1-m_N+k-1 - \frac{|\lambda|}{N}$$ lie in the set $\{-nr+1, \ldots, -nr+n\}$ and so contribute
toward sum $\fsum - krn$.
However the above makes sense not only for $r \in \ZZ$ but in the case
$kr \in \ZZ$. The entries in the first row that start on principal 
diagonal labeled $-\frac{|\lambda|}{N}$ correspond to left shift
by $r = \frac{m_1-m_N}{k}$.
Considering now all $N$ rows, the filling sum of this rectangle
is $\fsum - \sum_i k(\frac{m_i-m_N}{k})n = \fsum - |\lambda|n$,
and so our $p = -|\lambda|$ as in \eqref{eq-pSLdiagi}, as desired.
\end{proof}


\begin{figure}
\begin{tikzpicture}[scale=0.5]
\draw (0,12) node {$\lambda$};
\draw (7,12) node {$\T\in\Skk$};
\draw (14,12) node {$\Per(\T)$};
\draw (22,12) node {$\wtR(\Per(\T))$};

\draw (0,9.5) node {$\left(\begin{array}{c} 1\\0 \end{array}\right)$};
\draw (22,9.5) node {$(t^2,t^4,t^{-2},t^0)$};

\begin{scope}[shift={(5,8.5)}]
 \draw [thin,fill=gray!10] (0,1) rectangle (1,2);
\draw [thin,fill=red!10] (2,0) rectangle (3,1);
\draw[step=1,blue] (0,0) grid (2,1);
\draw[step=1,blue] (1,1) grid (3,2);
\draw (1.5,1.5) node[red] {$1$};
\draw (2.5,1.5) node[red] {$2$};
\draw (0.5,0.5) node[green!80!black] {$3$};
\draw (1.5,0.5) node[green!80!black] {$4$};
\draw [dashed,red] (0,2) -- (2,0);
\draw (-0.5,2.5) node[red]  {$0$};
\end{scope}


\begin{scope}[shift={(11,8.5)}]
\draw [very thick] (2,0) rectangle (4,2);
\draw[step=1, very thick] (2,0) grid (4,2);
\draw[step=1, gray] (0,0) grid (2,2);
\draw[step=1, gray] (4,0) grid (6,2);
\draw (0.5,1.5) node[gray]  {$-6$};
\draw (1.5,1.5) node[gray]  {$-3$};
\draw (2.5,1.5) node[black]  {$-2$};
\draw (3.5,1.5) node[black] {$1$};
\draw (4.5,1.5) node[gray] {$2$};
\draw (5.5,1.5) node[gray] {$5$};
\draw (6.5,1.5) node[gray] {$6$};

\draw (0.5,0.5) node[gray] {$-1$};
\draw (1.5,0.5) node[gray] {$0$};
\draw (2.5,0.5) node[black]  {$3$};
\draw (3.5,0.5) node[black]  {$4$};
\draw (4.5,0.5) node[gray]  {$7$};
\draw (5.5,0.5) node[gray]  {$8$};
\draw (6.5,0.5) node[gray]  {$11$};

\draw [dashed,blue] (1,3) -- (5,-1);
\draw (0.5,3.5) node[blue]  {$0$};
\end{scope}

\draw (0,5.5) node {$\left(\begin{array}{r} 0\\-1 \end{array}\right)$};
\draw (22,5.5) node {$(t^0,t^2,t^{-4},t^{-2})$};

\begin{scope}[shift={(5,4.5)}]
 \draw [thin,fill=gray!10] (0,1) rectangle (1,2);
\draw [thin,fill=red!10] (2,0) rectangle (3,1);
\draw[step=1,blue] (0,0) grid (2,1);
\draw[step=1,blue] (1,1) grid (3,2);
\draw (1.5,1.5) node[red] {$1$};
\draw (2.5,1.5) node[red] {$2$};
\draw (0.5,0.5) node[green!80!black] {$3$};
\draw (1.5,0.5) node[green!80!black] {$4$};
\draw [dashed,red] (0,2) -- (2,0);
\draw (-0.5,2.5) node[red]  {$-1$};
\end{scope}


\begin{scope}[shift={(11,4.5)}]
\draw [very thick] (2,0) rectangle (4,2);
\draw[step=1, very thick] (2,0) grid (4,2);
\draw[step=1, gray] (0,0) grid (2,2);
\draw[step=1, gray] (4,0) grid (6,2);
\draw (0.5,1.5) node[gray]  {$-3$};
\draw (1.5,1.5) node[gray]  {$-2$};
\draw (2.5,1.5) node[black]  {$1$};
\draw (3.5,1.5) node[black] {$2$};
\draw (4.5,1.5) node[gray] {$5$};
\draw (5.5,1.5) node[gray] {$6$};
\draw (6.5,1.5) node[gray] {$9$};

\draw (0.5,0.5) node[gray] {$0$};
\draw (1.5,0.5) node[gray] {$3$};
\draw (2.5,0.5) node[black]  {$4$};
\draw (3.5,0.5) node[black]  {$7$};
\draw (4.5,0.5) node[gray]  {$8$};
\draw (5.5,0.5) node[gray]  {$11$};
\draw (6.5,0.5) node[gray]  {$12$};

\draw [dashed,blue] (1,3) -- (5,-1);
\draw (0.5,3.5) node[blue]  {$0$};
\end{scope}

\draw (0,1.5) node {$\left(\begin{array}{c} 2\\1 \end{array}\right)$};
\draw (22,1.5) node {$(t^4,t^6,t^{0},t^2)$};

\begin{scope}[shift={(5,0.5)}]
 \draw [dotted,fill=gray!10] (-1,0) rectangle (0,1);
 \draw [dotted,fill=gray!10] (-1,1) rectangle (0,2);
 \draw [thin,fill=gray!20] (0,1) rectangle (1,2);
\draw [thin,fill=red!20] (2,0) rectangle (3,1);
\draw [dotted,fill=red!10] (3,0) rectangle (4,1);
\draw [dotted,fill=red!10] (3,1) rectangle (4,2);
\draw[step=1,blue] (0,0) grid (2,1);
\draw[step=1,blue] (1,1) grid (3,2);
\draw (1.5,1.5) node[red] {$1$};
\draw (2.5,1.5) node[red] {$2$};
\draw (0.5,0.5) node[green!80!black] {$3$};
\draw (1.5,0.5) node[green!80!black] {$4$};
\draw [dashed,red] (0,2) -- (2,0);
\draw (-0.5,2.5) node[red]  {$1$};
\end{scope}


\begin{scope}[shift={(11,0.5)}]
\draw [very thick] (2,0) rectangle (4,2);
\draw[step=1, very thick] (2,0) grid (4,2);
\draw[step=1, gray] (0,0) grid (2,2);
\draw[step=1, gray] (4,0) grid (6,2);
\draw (0.5,1.5) node[gray]  {$-7$};
\draw (1.5,1.5) node[gray]  {$-6$};
\draw (2.5,1.5) node[black]  {$-3$};
\draw (3.5,1.5) node[black] {$-2$};
\draw (4.5,1.5) node[gray] {$1$};
\draw (5.5,1.5) node[gray] {$2$};
\draw (6.5,1.5) node[gray] {$5$};

\draw (0.5,0.5) node[gray] {$-4$};
\draw (1.5,0.5) node[gray] {$-1$};
\draw (2.5,0.5) node[black]  {$0$};
\draw (3.5,0.5) node[black]  {$3$};
\draw (4.5,0.5) node[gray]  {$4$};
\draw (5.5,0.5) node[gray]  {$7$};
\draw (6.5,0.5) node[gray]  {$8$};

\draw [dashed,blue] (1,3) -- (5,-1);
\draw (0.5,3.5) node[blue]  {$0$};
\end{scope}

\end{tikzpicture}
\caption{Here $G=GL_2, k=2$.  The fundamental rectangle of $\Per(\T)$ is chosen so that the $0$th diagonal matches that of $\T\in\Skk$.}
\label{fig-psiexample}
\end{figure}

\begin{proposition}
\label{prop-psi-bijection}
The maps $\Per$ and $\SPer$ are bijections.
\end{proposition}
\begin{proof}

The inverse map $\Per^{-1}$ can be described as follows.  Given $R\in \psytk$
the boxes filled with $\{1,2,\ldots, n\}$  form a skew diagram $\cT$. 
Since the box
$k$ steps to the right of $\squarebox{\scriptstyle{i}}$ is
$\squarebox{\scriptstyle{i\!+\!n}}$, 
$\cT$ has shape $(\gamma + (k^N))/\gamma$ for some partition with
$\gamma_N=0$.
The principal diagonal of $\gamma$ has some diagonal label $r$.
Then set  $\lambda = (\sum_i \gamma_i \ep{i}) + r \wdet$.
 We may consider this diagonal-labeled  skew diagram $\cT$ to be in $\Skk$.

We do similarly for $\SPer^{-1}$, setting 
$\lambda = \sum_{i=1}^{N-1} \gamma_i \e{i}$.
Having already checked the diagonal labels for $\SPer$ are well-defined,
 the principal diagonal will already have inherited label $-\frac{|\lambda|}
{N}$.

\end{proof}

For an illustration of $\Per$ in the $\GL$ case see Figure \ref{fig-psiexample}.
In that example, note the skew tableaux are only differentiated
by their diagonal labels and likewise for the periodic tableaux.
For an example of $\SPer$ in the $\SL$ case,
 compare Figure \ref{fig-sl3-diagonals-skew}
to Figure \ref{fig-sl-diagonals-periodic}.

\subsection{The isomorphism type of $F_n^G(\cO_\q(G))$}

We are finally ready to prove our main theorem:

\begin{theorem}\label{thm-Fn-isom}
Let $\uu$ be a looped walk in $\gLatp$, respectively $\sLatp$.
\begin{enumerate}

\item Each subspace $L_{\underline{u}}$ is moreover a simultaneous $\YZ$-weight vector; for any $v\in L_{\underline{u}}$, we have:
$$Y_i v = t^{2\diagRi{\Tab(\underline{u})}{i} } v,
\qquad Z_i v = t^{2\diagRi{\STab(\underline{u})}{i} } v.$$

\item We have isomorphisms of $\HH_{q,t}(GL_N)$-modules, and $\HH_{q,t}(SL_N)$-modules, respectively:
$$F_n^{\GL}(\cO_q(GL_N))\cong L(k^N), \qquad F_n^{\SL}(\cO_q(SL_N))\cong \overline{L}(k^N).$$
\end{enumerate}
\end{theorem}

\begin{proof}  
For $\GLN$, all that remains is to simplify the exponent of $\q$ appearing in Theorem \ref{thm-Yscalar}.
 We compute first for $\GLN$:
\begin{align*}
&\!\!\!\!\!\!\!\!\langle u_i + 2 \rho, u_i \rangle -
\langle u_{i-1} + 2 \rho, u_{i-1} \rangle -
\langle \ep{1} + 2 \rho, \ep{1} \rangle\\
&= \langle u_{i-1}+\ep{\delta_i(\uu)} + 2\rho,
	u_{i-1}+\ep{\delta_i(\uu)} \rangle
- \langle u_{i-1}+2\rho,u_{i-1} \rangle
- \langle \ep{1}+2\rho,\ep{1}\rangle
\\
&=
\langle  u_{i-1},\ep{\delta_i(\uu)} \rangle
+ \langle 2\rho, \ep{\delta_i(\uu)} \rangle 
+\langle \ep{\delta_i(\uu)}, u_{i-1} \rangle
+\langle \ep{\delta_i(\uu)}, \ep{\delta_i(\uu)} \rangle
-\langle \ep{1}, \ep{1} \rangle
- \langle 2\rho,\ep{1}\rangle
\\
&=
2\left( 
\langle  u_{i-1},\ep{\delta_i(\uu)} \rangle
+ \langle \rho, \ep{\delta_i(\uu)} \rangle 
-\langle \rho, \ep{1} \rangle
\right)
\\
&= 2(\text{the label of the diagonal on which  $\squarebox{\scriptstyle{i}}$
lies in $\Tab(\uu)$}) , \text{ by \eqref{eq-form-diag-gl} }
\\ 
&= 2 \, \diag_{ \Per(\Tab(\uu))}(i).
\end{align*}
Hence we may re-write the scalar \eqref{Y-scalar} by which $Y_i$
$= \phi(\brY_i)$ acts,
as
$$\q^{2\diagRi{\Tab(\uu)}{i}} = t^{2\diagRi{\Tab(\uu)}{i}}.$$
Having matched their support,
the isomorphism $F_n^{\GL}(\cO_\q(GL_N))\cong L(k^N)$ 
then follows from Theorem \ref{thm-dahaLkN} and Corollary \ref{cor-Yssl}.

\medskip

In the case $G=\SLN$, a similar computation gives:
\begin{align*}
&\!\!\!\!\!\!\!\!\langle u_i + 2 \rho, u_i \rangle -
\langle u_{i-1} + 2 \rho, u_{i-1} \rangle -
\langle \e{1} + 2 \rho, \e{1} \rangle\\
&=
2\left( 
\langle  u_{i-1},\e{\delta_i(\uu)} \rangle
+ \langle \rho, \e{\delta_i(\uu)} \rangle 
-\langle \rho, \e{1} \rangle
\right)
\\
&=2\left( m_{\delta_i(\uu)}-m_N - \frac{|u_{i-1}|}{N}-\delta_i(\uu)+1 
\right)
\\
&=2\left(
- \frac{|u_{0}| +i-1}{N}
+m_{\delta_i(\uu)} -m_N + 1 - \delta_i(\uu)
\right)
\\
&= 2\frac{1-i}{N} +
2(\text{the label of the diagonal on which  $\squarebox{\scriptstyle{i}}$
lies in $\STab(\uu)$}) 
\\ 
&= 2\frac{1-i}{N} +
 2 \, \diagRi{ \SPer(\STab(\uu))}{i}.
\end{align*}
The second equality follows from 
\eqref{eq-form-on-lambda}, where we write
 $u_{i-1} = \sum_{j=1}^{N} m_j \e{j}$.
Since $m_j-m_N$ is the length of the $j$th row of $u_{i-1}$,
$\STab(\uu)$ places $\squarebox{\scriptstyle{i}}$ in column
$m_{\delta_i(\uu)} -m_N+ 1$ and row $\delta_i(\uu)$.
Recall from Section \ref{sec-typeA}
that the principal
diagonal of $\STab(\uu)$ is labeled $- \frac{|u_{0}|}{N}$.

Hence we may re-write the scalar \eqref{Y-scalar} 
by which $Z_i$ $=\phi(\Sq^{2(i-1)}\brY_i)$ acts,
as
$$\Sq^{2(i-1)}\q^{2 \frac{1-i}{N}}\q^{2\diagRi{\SPer\STab(u)}{i}}
=(\q^{\frac 1N})^{2(i-1)}\q^{2 \frac{1-i}{N}}\q^{2\diagRi{\SPer\STab(u)}{i}}
= t^{2\diagRi{\SPer\STab(u)}{i}}$$
since we have specialized $t=\q, \Sq=\q^{1/N}$. 

\medskip

This completely determines the support of $F_n^{\SL}(\cO_q(SL_N))$.  Comparing with Theorems \ref{thm-sLkN} and \ref{thm-sMkN} we see that:
$$ F_n^{\SL}(\cO_q(SL_N)) \cong \overline{L}(k^N)^a,$$
for some $a \in \K^\times$ with $a^n =1$.

In order to determine the twist parameter $a$, it suffices to consider a weight vector whose weight $\wtR(\bR_0)$ is stabilized by $\pi^N$:   the constant by which $\pi^N$ scales such a vector is $a^N$, and this determines the isomorphism type of the twist as in Theorem \ref{thm-sMkN}.  Hence, in order to conclude that $a=1$, it suffices to check that $\pi^N v_{\bR_0} = v_{\bR_0}$ with $\bR_0$ as in Figure \ref{stabilizedR}.

We note that $\bR_0 = \SPer(\STab(\uu_0))$ where $\uu_0$ is the looped walk at $\lambda=0$ whose steps satisfy $\delta_i(\uu_0) = i \bmod N$. In other words, $\uu_0$ is formed from the unique $N$-step looped walk $\uu$ at $\lambda=0$, concatenated with itself $k$ times.  We have:
$$v_{\bR_0} \in L_{\uu_0} = \underbrace{W_0^N \otimes \cdots \otimes 
W_0^N \otimes W_0^N}_{k} \subseteq W_0^{kN} = W_0^n,$$
where we recall that $W_0^N$ denotes the space $\Hom(\mathbf{1},V^{\ot N})$ of invariants in $V$.
It was shown in \cite{Jordan2008} that
$\brX_1 \cdots \brX_N|_{W_0^N} = (\nu^{-1}|_{V})^N \q^{N(\frac 1N -N)},$ hence $$\brX_1 \cdots \brX_N v_{\bR_0} = \q^{N(\frac 1N -N)} v_{\bR_0}.$$
We observe that $\phi(\Yprod^{-N} \brX_1 \cdots \brX_N \brT_w) = \pi^N,$
where $w$ is the length $N^2(k-1)$ permutation obtained by raising an $n$-cycle to the $N$th power.  
This sends
\begin{gather*}
\underset{k}{W_0^N} \otimes \underset{k-1}{W_0^N} \otimes  \cdots  \otimes 
 \underset{1}{W_0^N}
\qquad
\text{ to }
\qquad
\underset{1}{W_0^N} \otimes \underset{k}{W_0^N} \otimes  \cdots  \otimes 
 \underset{2}{W_0^N},
\end{gather*}
and in fact $\brT_w$ acts trivially on $ v_{\bR_0} $.  This computation is also consistent with 
$\q^{N(k-1)} \q^{(-\frac 1N) \ell(w)} =1$
in the $GL$ setting where the $(k-1)$ crossings of $\det$
contribute the first term.

Hence as we are taking $\Yprod = \Sq^{1-N^2} = \q^{1/N-N}$, we have
$$\pi^N v_{\bR_0} = \Yprod^{-N} \q^{N(\frac 1N -N)} v_{\bR_0}
 = \q^{-N(1/N-N)}\q^{N(1/N-N)} v_{\bR_0} = v_{\bR_0}.$$
\end{proof}

As a consequence of Theorem \ref{thm-Fn-isom}, we can give a combinatorial description of the action of $\pi$ on the weight basis of $F_n^G(\Oq)$, as follows.
\begin{definition}
 Given a looped walk $\underline{u}$ of length $n$, we may construct a new looped walk $\varpi(\underline{u})$,  by setting
\begin{align*}
\delta_i(\varpi(\uu)) = \delta_{i-1}(\uu) \quad (2\le i \le n),\qquad &\delta_1(\varpi(\uu)) = \delta_{n}(\uu),\\
\varpi(\underline{u})_i = \underline{u}_{i-1} \quad (1\le i \le n),\qquad &\varpi(\underline{u})_0 = \varpi(\underline{u})_{1} - \ep{\delta_n(\uu)}.
\end{align*}
\end{definition}
In other words, $\varpi(\underline{u})$ is the looped walk ending at $\underline{u}_{n-1}$ instead of $\underline{u}_n$. 

 We note that $\varpi^n(\uu) = \uu$ when $\uu$ is
 a looped walk of length $n$ in $\sLatp$, while 
in $\gLatp$ we have $\varpi^n(\uu)_i =\uu_i - k \wdet$.

\begin{corollary}
We have $\pi(W^n_\lambda) \subset \bigoplus_i W^n_{\lambda+\epsilon_i}$, and moreover
$$\pi L_{\underline{u}} = L_{\varpi(\underline{u})}.$$
\end{corollary}

While the corollary follows immediately from $\YZ$-semisimplicity of $F_n^G(\cO_\q(G))$ and the results of Section \ref{sec-Yssl}, it would not be straightforward to prove in purely Lie theoretic terms.  One can also directly deduce the action of the $X_i$ on the module combinatorially, in a similar fashion.

\bibliographystyle{plain}

\end{document}